\documentclass{IEEEtran}
\usepackage{amssymb,amsmath,amsthm,graphicx,subfigure}
\usepackage{algpseudocode,cases,booktabs}
\usepackage{tikz}
\usetikzlibrary{arrows.meta}

\newcommand{\bracket}[1]{\ensuremath{\left[ #1 \right]}}
\newcommand{\braces}[1]{\ensuremath{\left\{ #1 \right\}}}
\newcommand{\parenth}[1]{\ensuremath{\left( #1 \right)}}

\newcommand{\refeqn}[1]{(\ref{eqn:#1})}
\newcommand{\reffig}[1]{Figure \ref{fig:#1}}

\newcommand{\trs}[1]{\mathrm{tr}\ensuremath{[#1]}}

\newcommand{\deriv}[2]{\ensuremath{\frac{\partial #1}{\partial #2}}}
\newcommand{\dderiv}[2]{\ensuremath{\dfrac{\partial #1}{\partial #2}}}
\newcommand{\SO}{\ensuremath{\mathsf{SO(3)}}}

\newcommand{\so}{\ensuremath{\mathfrak{so}(3)}}

\renewcommand{\Re}{\ensuremath{\mathbb{R}}}

\renewcommand{\d}{\ensuremath{\mathfrak{d}}}
\newcommand{\Sph}{\ensuremath{\mathsf{S}}}

\DeclareMathOperator*{\argmax}{arg\,max}
\DeclareMathOperator*{\argmin}{arg\,min}

\date{}

\newtheorem{definition}{Definition}[section]
\newtheorem{lem}{Lemma}[section]
\newtheorem{prop}{Proposition}[section]
\newtheorem{remark}{Remark}[section]
\newtheorem{theorem}{Theorem}[section]

\graphicspath{{./Figs/}}

\title{Bayesian Attitude Estimation with the Matrix Fisher Distribution on SO(3)}
\author{Taeyoung Lee%
\thanks{Taeyoung Lee, Mechanical and Aerospace Engineering, George Washington University, Washington DC 20052 {\tt tylee@gwu.edu}}
\thanks{This research has been supported in part by NSF under the grants CMMI-1243000, CMMI-1335008, and CNS-1337722.}
}

\begin{document}

\maketitle

\begin{abstract}
This paper focuses on a stochastic formulation of Bayesian attitude estimation on the special orthogonal group. In particular, an exponential probability density model for random matrices, referred to as the matrix Fisher distribution is used to represent the uncertainties of attitude estimates and measurements in a global fashion. Various stochastic properties of the matrix Fisher distribution are derived on the special orthogonal group, and based on these, two types of intrinsic frameworks for Bayesian attitude estimation are constructed. These avoid complexities or singularities of the attitude estimators developed in terms of quaternions. The proposed approaches are particularly useful to deal with large estimation errors or large uncertainties for complex maneuvers to obtain accurate estimates of the attitude. 
\end{abstract}

\section{Introduction}

Attitude estimation has been widely studied with various filtering approaches and assumptions~\cite{CraMarJGCD07}. One of the unique challenges is that the attitude dynamics evolve on a compact, nonlinear manifold, referred to as the special orthogonal group. Attitude is often parameterized by three dimensional coordinates to develop an estimator. However, it is well known that minimal, three-parameter attitude representations, such as Euler-angles or modified Rodriguez parameters, suffer from singularities. They are not suitable for large angle rotational maneuvers, as the type of parameters should be switched in the vicinity of singularities. 

Quaternions are another popular choice of attitude representations. There is a wide variety of extended Kalman filters and unscented Kalman filtered developed in terms of quaternions for attitude estimation~\cite{CraMarJGCD97,PsiJGCD00,CraMarAJGCD03}. Quaternions do not exhibit singularities in representing attitude, but as the configuration space of quaternions, namely the three-sphere, double covers the special orthogonal group, there exists ambiguity that a single attitude is represented with two antipodal points on the three-sphere. Furthermore, the covariance defined in terms of quaternions exhibit singularity due to the unit-length constraint. To overcome this, various minimal attitude representations, such as Rodrigues parameters have been used to represent errors in quaternions-based attitude filters~\cite{MarJGCD03}. As such, these approaches are not suitable for large estimation errors.


Instead, attitude observers have been designed directly on the special orthogonal group to avoid both singularities of local coordinates and the ambiguity of quaternions. These include  complementary filters~\cite{MahHamITAC08}, a robust filter~\cite{SanLeeSCL08}, and a global attitude observer~\cite{WuKauPICDC15}. These approaches construct deterministic attitude observers on the special orthogonal group in an intrinsic fashion, and particularly in~\cite{MahHamITAC08,WuKauPICDC15}, asymptotic stability to the true attitude is rigorously shown via Lyapunov stability analysis. They are also robust against fixed gyro bias, and the efficacy of the complementary filter in~\cite{MahHamITAC08} has been illustrated by various numerical examples and experiments. However, these deterministic approaches should be distinguished from the probabilistic Bayesian attitude estimation problem, considered in this paper, that focuses on estimating the probabilistic distribution of the attitude uncertainty. While it is challenging to show stochastic stability of Bayesian estimators in general, they yield a probability density that completely describes the stochastic properties of the estimated attitude, including the degree of confidence or the level of the knowledge.



While they have been relatively unpopular in engineering communities, probability models and stochastic processes on a manifold have been studied in~\cite{Eme89,Elw82,MarJup99,Chi03}. In particular, earlier works on attitude estimation on the special orthogonal group include~\cite{LoEshSJAM79}, where a probability density function is expressed using noncommutative harmonic analysis~\cite{ChiKya01}. This idea of representing a probability density with harmonic analysis has been applied for uncertainty propagation and attitude estimation~\cite{ParKimP2IICRA05,MarPMDSAS05,LeeLeoPICDC08}. Using noncommutative harmonic analysis, an arbitrary probability density function on a Lie group can be defined globally up to any desired accuracy. Also, the Fokker-Planc k equation becomes transformed into ordinary differential equations via harmonic analysis, thereby providing an intrinsic solution for the attitude uncertainty propagation and attitude estimation. However, in practice, computing Fourier transforms on the special orthogonal group is infeasible for realtime implementation, as they involve integration of complicated unitary representations. Also, the order of Fourier transform should be increased as the estimated attitude distribution becomes more concentrated, thereby elevating the computational load as the estimator converges. For example, parallel computing has been applied for attitude uncertainty propagation with noncommutative harmonic analysis~\cite{LeeLeoPICDC08}.

Another notable recent work includes filtering techniques and measurement models developed in terms of exponential coordinates~\cite{ParLiuR08,Chi12,ChiKobPICDC14,LonWolRSV13}. This is perhaps the most natural approach to develop an estimator formally on an abstract Lie group, while taking advantages of the fact that the lie algebra is essentially a linear space. The limitation is that the exponential map is a local diffeomorphism around the identity element of the Lie group, and as such, the issue of a singularity remains.

This paper aims to construct probabilistic Bayesian attitude estimators on the special orthogonal group, while avoiding complexities of harmonic analysis and singularities of exponential coordinates. We use a specific form of the probability density, namely the matrix Fisher distribution~\cite{DowB72,KhaMarJRSSS77}, to represent uncertainties in the estimates of attitudes. The matrix Fisher distribution is a compact form of an exponential density model developed for random matrices, and the properties of the matrix Fisher distributions have been studied in directional statistics~\cite{MarJup99,Chi03}. When applied to the special orthogonal group, the matrix Fisher distribution is defined by 9 parameters, and therefore, it is comparable to the Gaussian distribution in $\Re^3$ that is completely defined by the three dimensional mean, and the six dimensional covariance. The first part of this paper is devoted to deriving various stochastic properties of the matrix Fisher distribution on the special orthogonal group, such as mean, moments and cumulative distributions, to understand the shape of the distribution. 

Based on those properties, Bayesian attitude estimators are constructed assuming that the uncertainties in attitude are represented by the matrix Fisher distribution. More explicitly, two types of propagation techniques are proposed. In the first method, an analytic expression is derived for the propagated first moment of the attitude uncertainty, and a propagated matrix Fisher distribution is constructed to match the moment. In the second method, an unscented transform is proposed for the matrix Fisher distribution on the special orthogonal group. For the measurement update step of Bayesian estimator, it is shown that the a posteriori distribution conditioned by arbitrary number of attitude measurements or direction measurements, follows a matrix Fisher distribution under mild assumptions. 

Combining the proposed propagation techniques with the measurement update, probabilistic attitude estimators are formulated globally on the special orthogonal group. These completely avoid singularities of attitude representations or covariance arising in other quaternion-based filters, and they are in contrast to deterministic, Luenberg-like attitude observers that do not incorporate uncertainties. The matrix Fisher distribution is rich enough to represent various attitude uncertainties, while avoiding complexities and challenges of non-commutative harmonic analysis in practice. It is also demonstrated that the proposed estimator exhibits excellent convergence properties for challenging cases with large initial estimation errors and large uncertainties.

In short, this paper presents an intrinsic, but practical formulation of stochastic attitude estimation on the special orthogonal group, which is particularity useful for challenging scenarios of large uncertainties and large errors. There is a great potential of applying the proposed matrix Fisher distribution on the special orthogonal group to various other stochastic analysis, such as stochastic optimal control or parameter estimation in attitude dynamics. 


This paper is organized as follows. The matrix Fisher distributions on the special orthogonal group and several stochastic properties are introduced at Section \ref{sec:MF}. Two types of attitude estimator are proposed in Section \ref{sec:FAE} and \ref{sec:UAE}, respectively, followed by numerical examples and conclusioins. 

\section{Matrix Fisher Distribution on $\SO$}\label{sec:MF}

Directional statistics deals with statistics for unit-vectors and rotations in $\Re^n$~\cite{MarJup99,Chi03}, where various  probability distributions on nonlinear compact manifolds are defined, and statistical analysis, such as inference and regressions are formulated in manifolds. 
In particular, the matrix Fisher distribution (or von Mises-Fisher matrix distribution) is a compact exponential density for random matrices introduced in~\cite{DowB72,KhaMarJRSSS77}. Interestingly, most of the prior work on the matrix Fisher distributions in directional statistics are developed for the Stiefel manifold, $\mathsf{V}_k(\Re^n)=\{X\in\Re^{n\times k}\,|\, XX^T=I_{n\times n}\}$, i.e., the set of $k$ orthonormal vectors in $\Re^n$. 

The configuration manifold for the attitude dynamics of a rigid body is the three-dimensional special orthogonal group, 
\begin{align*}
\SO = \{R\in\Re^{3\times 3}\,|\, R^TR=I_{3\times 3},\,\mathrm{det}[R]=1\},
\end{align*}
where each rotation matrix corresponds the linear transformation of the representation of a vector from the body-fixed frame to the inertial frame. This section focuses on formulating the matrix Fisher distribution on $\SO$, and deriving various stochastic properties. These will be utilized in the development of attitude estimators later.


Throughout this paper, the Lie algebra of $\SO$ is denoted by $\so=\{S\in\Re^{3\times 3}\,| S=-S^T\}$. The \textit{hat} map: $\wedge:\Re^3\rightarrow \so$ is an isometry between $\so$ and $\Re^3$ defined such that $\hat x = -(\hat x)^T$, and $\hat x y =x\times y$ for any $x,y\in\Re^3$. The inverse of the hat map is denoted by the \textit{vee} map: $\vee:\so\rightarrow\Re^3$. The two-sphere is the set of unit-vectors in $\Re^3$, i.e., $\Sph^2=\{q\in\Re^3\,|\,\|q\|=1\}$, and the $i$-th standard basis of $\Re^3$ is denoted by $e_i\in\Sph^2$ for $i\in\{1,2,3\}$.  The set of circular shifts of $(1,2,3)$ is defined as $\mathcal{I}=\{(1,2,3),(2,3,1),(3,1,2)\}$. The Frobenius norm of a matrix $A\in\Re^{3\times 3}$ is defined as $\|A\|_F=\sqrt{\trs{A^TA}}$. 

The subsequent developments require the modified Bessel function of the first kind~\cite{AbrSte65}. For convenience, the definition of the zeroth order function, and the first-order function are copied below. For any $x\in\Re$,
\begin{align}
I_0(x) &= \frac{1}{\pi}\int_0^\pi \exp(x\cos\theta)\,d\theta
=\sum_{n=0}^\infty \frac{(\frac{1}{2}x)^{2n}}{(n!)^2},\label{eqn:I0}\\
I_1(x) &= \frac{1}{\pi}\int_0^\pi \cos\theta\exp(x\cos\theta)\,d\theta=\frac{d}{dx}I_0(x),\label{eqn:I1}
\end{align}
which satisify
\begin{alignat}{2}
I_0(0)&=1,&\qquad 1\leq I_0(|x|)&=I_0(-|x|),\label{eqn:pI0}\\
I_1(0)&=0,&\qquad  0\leq I_1(|x|)&=-I_1(-|x|),\label{eqn:pI1}
\end{alignat}
i.e., $I_0(x)$ is an even-function of $x$ greater than or equal to $1$, and $I_1(x)$ is an odd-function passing through the first and the third quadrants. 

\subsection{Matrix Fisher Distribution on $\SO$}

The matrix Fisher distribution on $\SO$ is a member of exponential families that is defined by 9  parameters as follows. 

\begin{definition}\label{def:MF}
A random rotation matrix $R\in\SO$ is distributed according to a matrix Fisher distribution, if its probability density function is defined relative to the uniform distribution on $\SO$ as
\begin{align}
p(R)=\frac{1}{c(F)}\exp(\trs{F^T R}),\label{eqn:MF}
\end{align}
where $F\in\Re^{3\times 3}$, and $c(F)\in\Re$ is a normalizing constant defined such that $\int_{\SO} p(R)\, dR=1$. This is also denoted by $R\sim\mathcal{M}(F)$.
\end{definition}

The above definition implies that the normalizing constant is given by
\begin{align}
c(F) = \int_{\SO} \exp(\trs{F^T R}) dR. \label{eqn:cF}
\end{align}
To evaluate $c(F)$, and therefore $p(R)$, the measure $dR$ on $\SO$ should be defined explicitly. For the Lie group $\SO$, there is a bi-invariant measure, referred to as \textit{Haar} measure, which is unique up to scalar multiples~\cite{ChiKya01}. Throughout this paper, the Haar measure $dR$ is scaled such that $\int_{\SO} dR=1$. In other words, the uniform distribution on $\SO$ is given by $p(R)=1$ for the selected Haar measure. This is equivalent to state that \refeqn{MF} is defined relative to the uniform distribution on $\SO$. For instance, when $R\in\SO$ is parameterized by the 3--1--3 Euler angles, i.e., $R(\alpha,\beta,\gamma)=\exp(\alpha\hat e_3)\exp(\beta\hat e_1)\exp(\gamma\hat e_3)$ for $\alpha,\gamma\in[0,2\pi)$ and $\beta\in[0,\pi]$, the normalized Haar measure is given by
\begin{align*}
dR(\alpha,\beta,\gamma)=\frac{1}{8\pi^2} \sin\beta\, d\alpha\, d\beta\, d\gamma.
\end{align*}

While there are various approaches to evaluate the normalizing constant for the matrix Fisher distribution on the Stiefel manifold, few papers deal with calculation of the normalizing constant \refeqn{cF} for $\SO$. A method based on the gradient descent is introduced in~\cite{SeiShiJMA13}, which involves the numerical solution of multiple ordinary differential equations. The normalizing constant is expressed as a one-dimensional integration in~\cite{WooAJS93}, but the given result is erroneous as the change of volume over a certain transformation is not considered properly. We follow the approach of~\cite{WooAJS93}, to derive another corrected expression of the normalizing constant. It turns out that the normalizing constant $c(F)$ depends on the singular values of $F$ as discussed below.

Suppose the singular value decomposition of $F$ is given by
\begin{align}
F= U' S' (V')^T,\label{eqn:USV}
\end{align}
where $S'\in\Re^{3\times 3}$ is a diagonal matrix composed of the singular values $ s'_1\geq s'_2\geq s'_3\geq 0$ of $F$, sorted in descending order, and $U',V'\in\Re^{3\times 3}$ are orthonormal matrices. Since $(U')^TU'=(V')^TV'=I_{3\times 3}$, the determinant of $U'$ or $V'$ is $\pm 1$. Note the orthogonal matrices $U'$ and $V'$ are not necessarily rotation matrices in $\SO$, as their determinant is possibly $-1$. To resolve this, we introduce a \textit{proper} singular value decomposition as follows.

\begin{definition}{(\cite{MarJAS88})}\label{def:PSVD}
For a given $F\in\Re^{3\times 3}$, let the singular value decomposition be given by \refeqn{USV}. The `proper' singular value decomposition of $F$ is defined as
\begin{align}
F= U SV^T,\label{eqn:USVp}
\end{align}
where the rotation matrices $U,V\in\SO$, and the diagonal matrix $S\in\Re^{3\times 3}$ are 
\begin{align}
U &= U' \mathrm{diag}[1,\,1,\, \mathrm{det}[U']],\\
S & = \mathrm{diag}[s_1,\,s_2,\,s_3]=\mathrm{diag}[s'_1,\,s'_2,\, \mathrm{det}[U'V']s'_3],\\
V &= V' \mathrm{diag}[1,\,1,\, \mathrm{det}[V']].
\end{align}
\end{definition}

The definition of $U$ and $V$ are formulated such that $\mathrm{det}[U]=\mathrm{det}[V]=+1$ to ensure $U,V\in\SO$. Note that the first two proper singular values $s_1,s_2$ are non-negative, but the last one $s_3$ could be negative when $\mathrm{det}[U'V'] =-1$ and $s_3' >0$. As $s_1',s_2',s_3'$ are sorted in descending order,
\begin{gather}
0\leq |s_3|\leq s_2\leq s_1,\label{eqn:s123}\\
0\leq s_2+s_3\leq s_3+s_1\leq s_1+s_2.\label{eqn:s12}
\end{gather}


\begin{theorem}\label{thm:C}
Suppose the proper singular value decomposition of $F$ is given by \refeqn{USVp}. The normalizing constant for the matrix Fisher distribution \refeqn{cF} satisfies the following properties.
\begin{itemize}
\item[(i)] $c(F)=c(S)=c(\mathrm{diag}[s_i,s_j,s_k])$ for any $(i,j,k)\in\mathcal{I}$.
\item[(ii)] $c(F)=c(F^T)=c(Q_0 F Q_1)$ for any $Q_0,Q_1\in\SO$.
\item[(iii)] $e^{-s_1-s_2+s_3}\leq c(S)\leq e^{s_1+s_2+s_3}$, where the inequalities become strict when $S\neq 0_{3\times 3}$.
\item[(iv)] $c(S)$ is evaluated by a one-dimensional integral as
\begin{align}
c(S) & = \int_{-1}^1 \frac{1}{2}I_0\!\bracket{\frac{1}{2}(s_i-s_j)(1-u)} \nonumber\\
&\quad \times I_0\!\bracket{\frac{1}{2}(s_i+s_j)(1+u)}\exp (s_ku)\,du,\label{eqn:cS}
\end{align}
for any $(i,j,k)\in\mathcal{I}$.
\item[(v)] The first order derivatives of $c(S)$ are given by
\begin{subequations}\label{eqn:dcS}
\begin{align}
\deriv{c(S)}{s_i} 
& = \int_{-1}^1 \frac{1}{4}(1-u)I_1\!\bracket{\frac{1}{2}(s_i-s_j)(1-u)}\nonumber\\
&\quad \times  I_0\!\bracket{\frac{1}{2}(s_i+s_j)(1+u)}\exp (s_ku)\nonumber\\
&\quad +\frac{1}{4}(1+u)I_0\!\bracket{\frac{1}{2}(s_i-s_j)(1-u)}\nonumber\\
&\quad \times  I_1\!\bracket{\frac{1}{2}(s_i+s_j)(1+u)}\exp (s_ku)\,du,\label{eqn:dcSa}\\
\deriv{c(S)}{s_j} & = \int_{-1}^1 -\frac{1}{4}(1-u)I_1\!\bracket{\frac{1}{2}(s_i-s_j)(1-u)}\nonumber\\
&\quad \times  I_0\!\bracket{\frac{1}{2}(s_i+s_j)(1+u)}\exp (s_ku)\nonumber\\
&\quad +\frac{1}{4}(1+u)I_0\!\bracket{\frac{1}{2}(s_i-s_j)(1-u)}\nonumber\\
&\quad \times  I_1\!\bracket{\frac{1}{2}(s_i+s_j)(1+u)}\exp (s_ku)\,du,\label{eqn:dcSb}\\
\deriv{c(S)}{s_k} & = \int_{-1}^1 \frac{1}{2}I_0\bracket{\frac{1}{2}(s_i-s_j)(1-u)} \nonumber\\
&\quad\times I_0\bracket{\frac{1}{2}(s_i+s_j)(1+u)}u\exp (s_ku)\,du,\label{eqn:dcSc}
\end{align}
\end{subequations}
for any $(i,j,k)\in\mathcal{I}$, and they satisfy the following property
\begin{gather}
0\leq \deriv{c(S)}{s_i}+\deriv{c(S)}{s_j}\quad\mbox{for any $(i,j,k)\in\mathcal{I}$}\label{eqn:dcSij},\\
0\leq \left| \deriv{c(S)}{s_3}\right| \leq \deriv{c(S)}{s_2} \leq \deriv{c(S)}{s_1}.\label{eqn:dcS12}
\end{gather}

\item[(vi)] The second order derivatives of $c(S)$ are given by
\begin{align}
\frac{\partial^2 c(S)}{\partial s_i^2} & = \int_{-1}^1 \frac{1}{2}I_0\bracket{\frac{1}{2}(s_j-s_k)(1-u)}\nonumber\\
&\times I_0\bracket{\frac{1}{2}(s_j+s_k)(1+u)}u^2\exp (s_iu)\,du,\label{eqn:ddcSii}\\
\frac{\partial^2 c(S)}{\partial s_i\partial s_j} & = \int_{-1}^1 \frac{1}{4}I_1\bracket{\frac{1}{2}(s_j-s_k)(1-u)}\nonumber\\
&\times I_0\bracket{\frac{1}{2}(s_j+s_k)(1+u)}u(1-u)\exp (s_iu)\nonumber\\
& + \frac{1}{4}I_0\bracket{\frac{1}{2}(s_j-s_k)(1-u)} \nonumber\\
&\times I_1\bracket{\frac{1}{2}(s_j+s_k)(1+u)}u(1+u)\exp (s_iu)\,du,\label{eqn:ddcS}
\end{align}
for any $(i,j,k)\in\mathcal{I}$.
\end{itemize}
\end{theorem}
\begin{proof}
See Appendix~\ref{sec:PfC}.
\end{proof}

Theorem \ref{thm:C} implies that the normalizing constant of the matrix Fisher distribution on $\SO$ is evaluated by a one-dimensional integral, and it only depends on the proper singular values of $F$. Since \refeqn{cS} is invariant under the circular permutation of $(i,j,k)$, so is \refeqn{dcS}. For example, $\deriv{c(S)}{s_1}$ can be computed by any of \refeqn{dcSa}, \refeqn{dcSb}, or \refeqn{dcSc}. The second order derivatives \refeqn{ddcS} are derived from \refeqn{dcSc} for convenience, but other equivalent expressions for the second order derivatives can be certainly derived from \refeqn{dcSb} or \refeqn{dcSc}.

The first property yields the following transformation to obtain a shifted matrix Fisher distribution with a diagonal matrix parameter.

\begin{lem}\label{lem:Q}
Suppose $R\sim\mathcal{M}(F)$ for a matrix $F\in\Re^{3\times 3}$. The proper singular value decomposition of $F$ is given by \refeqn{USVp}. Define $Q=U^T R V\in\SO$. Then, $Q\sim \mathcal{M}(S)$.
\end{lem}
\begin{proof}
As the transformation from $R$ to $Q$ is volume preserving, the probability density of $Q$ is given by
\begin{align*}
p_Q(Q) &=\frac{1}{c(S)} \exp(\trs{F^TR(Q)}).
\end{align*}
Substituting $R=UQV^T$ and \refeqn{USVp}, we can show $p_Q(Q)\propto \exp(\trs{S^TQ})$, 
which follows $Q\sim\mathcal{M}(S)$. 
\end{proof}

\subsection{Moments and Mean Attitude}

Here, we characterize a matrix Fisher distribution via its moments and mean. The first two moments are derived as follows, for the case that the random rotation matrix is transformed such that the matrix parameter becomes diagonal as discussed in Lemma \ref{lem:Q}. 

\begin{theorem}\label{thm:Q}
Suppose $Q\sim\mathcal{M}(S)$ for a diagonal matrix $S=\mathrm{diag}[s_1,s_2,s_3]\in\Re^{3\times 3}$, satisfying \refeqn{s123}, \refeqn{s12}. Let $Q_{ij}$ be the $(i,j)$-th element of $Q$ for $i,j\in\{1,2,3\}$.
\begin{itemize}
\item[{(i)}] The first moment of each element of $Q$ is given by
\begin{align}
\mathrm{E}[Q_{ij}]
=\begin{cases}\dfrac{1}{c(S)}\dderiv{c(S)}{s_i}=\dderiv{\log c(S)}{s_i} & \mbox{if $i= j$},\\
0 & \mbox{otherwise},
\end{cases}\label{eqn:M1}
\end{align}
for $i,j\in\{1,2,3\}$.
\item[{(ii)}] The second moments of elements $Q$ are given by
\begin{align}
    \mathrm{E}[Q_{ii}Q_{jj}] = \dfrac{1}{c(S)}\dfrac{\partial^2 c(S)}{\partial s_i\partial s_j},\label{eqn:EQiijj}
\end{align}
for $i,j\in\{1,2,3\}$.
Also,
\begin{align}
    \mathrm{E}[Q_{ij}Q_{ij}] & = \frac{1}{(s_j^2-s_i^2)c(S)} \braces{ s_j \deriv{c(S)}{s_j}  - s_i \deriv{c(S)}{s_i} }, \label{eqn:EQijij}\\
    \mathrm{E}[Q_{ij}Q_{ji}] & = \frac{1}{(s_j^2-s_i^2)c(S)} \braces{ s_i \deriv{c(S)}{s_j}  - s_j \deriv{c(S)}{s_i} },\label{eqn:EQijji}
\end{align}
for $i\neq j\in\{1,2,3\}$.
Finally, $\mathrm{E}[Q_{ij}Q_{kl}]=0$ for any pair of $i,j,k,l\in\{1,2,3\}$ that does not belong to any of the above three cases. 
\end{itemize}
\end{theorem}
\begin{proof}
We first show the zero values of the moments. Let $\gamma_1,\gamma_2,\gamma_3\in\{-1,1\}$ be chosen such that $\gamma_1\gamma_2\gamma_3=1$, i.e.,
\begin{align}
(\gamma_1,\gamma_2,\gamma_3)\in\{(1,1,1),(1,-1,-1),(-1,-1,1),(-1,1,-1)\}.\label{eqn:li}
\end{align}
Define $\Gamma=\mathrm{diag}[\gamma_1,\gamma_2,\gamma_3]$. We can show $\Gamma=\Gamma^T=\Gamma^{-1}\in\SO$. Let $Z=\Gamma^T Q \Gamma\in\SO$ such that $Z_{ij}=\gamma_i\gamma_j Q_{ij}$ for any $i,j\in\{1,2,3\}$.  It is straightforward to show $Z\sim\mathcal{M}(S)$ since $\trs{S^TZ}=\sum_{i=1}^3 s_iZ_{ii} = \sum_{i=1}^3 s_i\gamma_i^2 Q_{ii}=\trs{S^TQ}$. Therefore,
\begin{align*}
\mathrm{E}[Q_{ij}]=\mathrm{E}[Z_{ij}]=\gamma_i\gamma_j\mathrm{E}[Q_{ij}],
\end{align*}
which implies that $\mathrm{E}[Q_{ij}]=0$ when $\gamma_i\gamma_j=-1$. In case of $i\neq j$, we can always choose pick one from \refeqn{li} such that $\gamma_i\gamma_j=-1$. For example, when $(\gamma_1,\gamma_2,\gamma_3)=(1,-1,-1)$, the above equation implies $\mathrm{E}[Q_{12}]=\mathrm{E}[Q_{21}]=\mathrm{E}[Q_{13}]=\mathrm{E}[Q_{31}]=0$. Similarly, we can show $\mathrm{E}[Q_{ij}]=0$ for any $i\neq j$.

Similarly, $\mathrm{E}[Q_{ij}Q_{kl}]=\gamma_i\gamma_j\gamma_k\gamma_l\mathrm{E}[Q_{ij}Q_{kl}]$. 
Suppose $(i=j\land k=l)$ or $(i=k\neq j=l)$ or $(i=l\neq j=k)$, i.e., $(i,j,k,l)$ follows the pattern of $(i,i,j,j)$, $(i,j,i,j)$, or $(i,j,j,i)$, where $i\neq j$ for the last two cases. 
Then $\gamma_i\gamma_j\gamma_k\gamma_l=1$ always.  
For the other cases, we can always select $(\gamma_1,\gamma_2,\gamma_3)$ from \refeqn{li} such that $\gamma_i\gamma_j\gamma_k\gamma_l=-1$, which implies $\mathrm{E}[Q_{ij}Q_{kl}]=0$. 
For instance, one can show $\mathrm{E}[Q_{11}Q_{12}]=0$ by selecting $(\gamma_1,\gamma_2,\gamma_3)=(1,-1,-1)$ such that $\gamma_1^3\gamma_2 = -1$. 

Next, we show the non-zero values. For $T\in\Re^{3\times 3}$, let the moment generating function of $Q$ be
\begin{align*}
M_Q(T) = \mathrm{E}[\exp(\trs{T^T Q})].
\end{align*}
Since $Q\sim \mathcal{M}(S)$, this is simplified as
\begin{align*}
M_Q(T)=\frac{1}{c(S)}\int_{\SO} \exp (\trs{(S+T)^T Q})\, dQ = \frac{c(S+T)}{c(S)}.
\end{align*}
The derivatives of the moment generating function at $T=0$ correspond to the moment, i.e.,
\begin{align*}
\mathrm{E}[Q_{ii}]=\frac{\partial}{\partial T_{ii}}\parenth{\frac{c(S+T)}{c(S)}}\bigg|_{T=0}
=\frac{1}{c(S)} \deriv{c(S)}{s_i}, 
\end{align*}
as $T_{ii}$ is a diagonal element of $T$.
Similarly, 
\begin{align*}
\mathrm{E}[Q_{ii}Q_{jj}]=\frac{\partial^2}{\partial T_{ii}\partial T_{jj}}\parenth{\frac{c(S+T)}{c(S)}}\bigg|_{T=0}
=\frac{1}{c(S)} \frac{\partial c(S)}{\partial s_i\partial s_j},
\end{align*}
which yields \eqref{eqn:EQiijj}.
Next, for $i\neq j$ consider
\begin{align*}
\mathrm{E}[Q_{ij}Q_{ij}]=\frac{\partial^2}{\partial T_{ij}^2}\parenth{\frac{c(S+T)}{c(S)}}\bigg|_{T=0}.
\end{align*}
Let the singular values of $S+T$ be $(\tilde s_1, \tilde s_2, \tilde s_3)$.
The above is expanded into
\begin{align*}
    \mathrm{E}[Q_{ij}Q_{ij}] & = \frac{1}{c(S)} \sum_{k,l=1}^3 \frac{\partial c(S+T)}{\partial \tilde s_k \partial \tilde s_l} \deriv{\tilde s_k}{T_{ij}}\deriv{\tilde s_l}{T_{ij}}\bigg|_{T=0} \\
                             & \quad + \frac{1}{c(S)}\sum_{k=1}^3 \deriv{c(S+T)}{\tilde s_k} \frac{\partial^2 \tilde s_k}{\partial \tilde T_{ij}^2}\bigg|_{T=0}.
\end{align*}
From~\cite{papadopoulo2000estimatinga}, the derivatives of the singular value are given by
\begin{align*}
    \deriv{\tilde s_k}{T_{ij}}\bigg|_{T=0} & = 0,\\
    \frac{\partial^2 \tilde s_k}{\partial T_{ij}^2}\bigg|_{T=0} & = \delta_{j,k} \frac{s_j}{s_j^2-s_i^2} - \delta_{i,k} \frac{s_i}{s_j^2 - s_i^2}.
\end{align*}
Substituting these yields \eqref{eqn:EQijij}. 
Equation \eqref{eqn:EQijji} can be shown similarly. 
\end{proof}

While the above theorem is developed for $Q$, it can be used to find the moments of $R$ as follows. 

\begin{lem}\label{lem:ER}
Suppose $R\sim\mathcal{M}(F)$ for a matrix parameter $F\in\Re^{3\times 3}$. Let the proper singular value decomposition of $F$ be given by \refeqn{USVp}. Then, the first moment of $R$ is given by
\begin{align}
\mathrm{E}[R]=U\mathrm{E}[Q]V^T,\label{eqn:ER}
\end{align}
where the diagonal matrix $\mathrm{E}[Q]$ is constructed by using \refeqn{M1} from $S$. Furthermore, \refeqn{ER} represents the proper singular value decomposition of $\mathrm{E}[R]$.
\end{lem}
\begin{proof}
According to Lemma \ref{lem:Q}, $U^T RV\sim\mathcal{M}(S)$. Therefore, $\mathrm{E}[Q]=\mathrm{E}[U^T RV]=U^T \mathrm{E}[R]V$, which yields \refeqn{ER}. Due to \refeqn{dcSij} and \refeqn{dcS12}, the diagonal elements of $\mathrm{E}[Q]$ serve as proper singular values, and $U,V\in\SO$. Therefore, \refeqn{ER} represents a proper singular value decomposition.
\end{proof}

As seen in \refeqn{ER}, the arithmetic mean of $R$ is defined by 9 parameters, and it does not necessarily belong to $\SO$. Consequently, we need to formulate the three-dimensional \textit{mean attitude} in the context of $\SO$. To avoid confusion, the arithmetic mean $\mathrm{E}[R]$ is referred to as the \textit{first moment} throughout this paper. Two types of the mean attitude are defined as follows. 
\begin{definition}
Let the probability density of $R\in\SO$ be $p(R):\SO\rightarrow\Re$. The max mean is defined as
\begin{align}
\mathrm{M}_{\max} [R]=\argmax_{R\in\SO} \{p(R)\},\label{eqn:Mmax}
\end{align}
and the minimum mean square error (MMSE) mean is defined as
\begin{align}
\mathrm{M}_{\mathrm{mse}} [R]=\argmin_{R\in\SO} \braces{\int_{\SO} \|R-\tilde R\|_F^2 \,p(\tilde R) d\tilde R}.\label{eqn:Mmse}
\end{align}
\end{definition}

For the matrix Fisher distribution, the above mean attitudes are equivalent and they can be obtained explicitly as follows.
\begin{theorem}\label{thm:MAtt}
Suppose $R\sim\mathcal{M}(F)$ with a matrix $F\in\Re^{3\times 3}$ with the proper singular value decomposition \refeqn{USVp}. The max mean and the minimum mean square mean of $R$ are identical and they are given by
\begin{align}
\mathrm{M}_{\max} [R]=\mathrm{M}_{\mathrm{mse}} [R]=UV^T\in\SO.\label{eqn:MAtt}
\end{align}
\end{theorem}
\begin{proof}
As discussed in \refeqn{trSQ} of Appendix \ref{sec:PfC}, the probability density of a matrix Fisher distribution is maximized when $U^T RV=I_{3\times 3}$, or $R=UV^T$, which is the max mean attitude. 

Next, from the definition of the Frobenius norm,
\begin{align*}
\|R-\tilde R\|_F^2 = \trs{(R-\tilde R)^T (R-\tilde R)} = \trs{2(I_{3\times 3}-R^T\tilde R)}.
\end{align*}
Substituting this into \refeqn{Mmse}, and using $\int_\SO p(\tilde R)d\tilde R=1$ and $\int_\SO \tilde R p(\tilde R)d\tilde R=\mathrm{E}[R]$,
\begin{align*}
\mathrm{M}_{\mathrm{mse}} [R]=\argmin_{R\in\SO} \braces{6-2\trs{R^T \mathrm{E}[R]}}.
\end{align*}
Substituting \refeqn{ER}, it is straightforward to show $\mathrm{M}_{\mathrm{mse}}$ is the value of $R$ that maximizes $\trs{R^T \mathrm{E}[R]}=\trs{R^T U\mathrm{E}[Q] V^T}=\trs{\mathrm{E}[Q] (V^TR^T U)}$. As discussed in Lemma \ref{lem:ER}, the diagonal elements of $\mathrm{E}[Q]$ serve as the proper singular values due to \refeqn{dcSij}. Therefore, similar with \refeqn{trSQ}, the above expression is maximized when $V^TR^T U=I_{3\times 3}$, or equivalently $R=UV^T$.
\end{proof}
	
\subsection{Cumulative Distribution for $\mathcal{M}(sI_{3\times 3})$}

From Theorem \ref{thm:Q}, the proper singular values of $F$ determines the dispersion or  concentration of the matrix Fisher distribution. In this subsection, we analyze the degree of dispersion of the matrix Fisher distribution quantitatively for the simplified cases when the matrix parameter $F$ is directly proportional to the identity matrix, i.e., $R\sim\mathcal{M}(sI_{3\times 3})$ for $s\geq 0$. 

\begin{theorem}
Suppose $R\sim\mathcal{M}(sI_{3\times 3})$ for $s\geq 0$. The probability density for $R$ is given by
\begin{align}
p(R) = \frac{1}{I_0(2s)-I_1(2s)} \exp (s\,(\trs{R}-1)),\label{eqn:MFs}
\end{align}
relative to the uniform distribution on $\SO$. For $\theta\in[0,\pi]$, define $E_\theta\in\SO$ be
\begin{align}
E_\theta = \{R\in\SO \,|\, \|(\log(R))^\vee\|\leq \theta\},
\end{align}
which is the set of the rotation matrices that can be obtained from a fixed-axis rotation of $I_{3\times 3}$ with a rotation angle less than or equal to $\theta$. The probability that $R$ belongs to $E_\theta$ is given by
\begin{align}
\mathrm{Prob}&[R\in E_\theta]\nonumber\\=
&\frac{1}{\pi(I_0(2s)-I_1(2s))} \int_{0}^\theta \exp(2s\cos\rho) (1-\cos\rho)   d\rho.\label{eqn:PrEtheta}
\end{align}
\end{theorem}
\begin{proof}
The rotation matrix is parameterized by the exponential map as
\begin{align*}
R(\eta)=\exp(\hat\eta)=I_{3\times 3}+\sin\|\eta\|\frac{\hat\eta}{\|\eta\|} +(1-\cos\|\eta\|)\frac{\hat\eta^2}{\|\eta\|^2},
\end{align*}
where $\eta\in\Re^3$. This can be interpreted as the rotation about the axis $\frac{\eta}{\|\eta\|}$ by the angle $\|\eta\|$. Since $\trs{R(\eta)}=1+2\cos\|\eta\|$, and $dR(\eta)= \frac{1-\cos\|\eta\|}{4\pi^2 \|\eta\|^2}$~\cite{ChiKya01}, the probability that $R$ belongs to $E_\theta$ is
\begin{align*}
&\mathrm{Prob}[R\in E_\theta]=\int_{R\in E_\theta}p(R(\eta))dR(\eta)\\
& = \frac{1}{c(sI_{3\times 3})}\int_{\|\eta\|\leq \theta}\exp(s(1+2\cos\|\eta\|)) \frac{1-\cos\|\eta\|}{4\pi^2 \|\eta\|^2}d\eta.
\end{align*}
To evaluate the integral, we parameterize $\eta$ via
\begin{align*}
\eta = [\rho\sin\phi \cos\psi,\, \rho \sin\phi\sin\psi,\, \rho\cos\phi]^T,
\end{align*}
where $\rho\in[0,\theta]$, $\phi\in[0,\pi]$, and $\psi\in[0,2\pi]$ are the radius, the co-latitude, and the longitude of $\eta$, respectively. Substituting this with $d\eta=\rho^2 \sin\phi\, d\rho d\psi d\phi$, 
\begin{align}
\mathrm{Prob}[R\in E_\theta] =\frac{e^s}{\pi c(sI_{3\times3})} \int_{0}^\theta \exp(2s\cos\rho) (1-\cos\rho)   d\rho.
\label{eqn:PrEtheta0}
\end{align}
Since $\mathrm{Prob}[R\in E_\pi]=\mathrm{Prob}[R\in \SO]=1$, the normalizing constant is given by
\begin{align*}
c(sI_{3\times 3}) & =  \frac{e^s}{\pi}\int_0^\pi \exp(2s\cos\rho) (1-\cos\rho)d\rho.
\end{align*}
From \refeqn{I0} and \refeqn{I1}, it reduces to
\begin{align}\label{eqn:csI}
c(sI_{3\times 3}) 
& = e^s (I_0(2s)-I_1(2s)).
\end{align}
Substituting this into \refeqn{MF} with $F=sI_{3\times 3}$ yield \refeqn{MFs}, and \refeqn{PrEtheta} is obtained from \refeqn{PrEtheta0}.
\end{proof}
\begin{remark}
Suppose $R\sim\mathcal{M}(sM)$ for $M\in\SO$ and $s\geq 0$, i.e., the three singular values of the matrix parameter $F$ are identical, and the max (or MMSE) mean attitude is $M$. According to Lemma \ref{lem:Q}, $M^T R,RM^T\sim\mathcal{M}(s I_{3\times 3})$. Therefore, the probability that the angle between $R$ and $M$ is less than $\theta$ for $\theta\in[0,\pi]$ is also given by \refeqn{PrEtheta}. 
\end{remark}

\begin{figure}
\centerline{
	\subfigure[$\mathrm{Prob}{[} R\in E_\theta {]}$ for varying $s$]{\includegraphics[width=0.8\columnwidth]{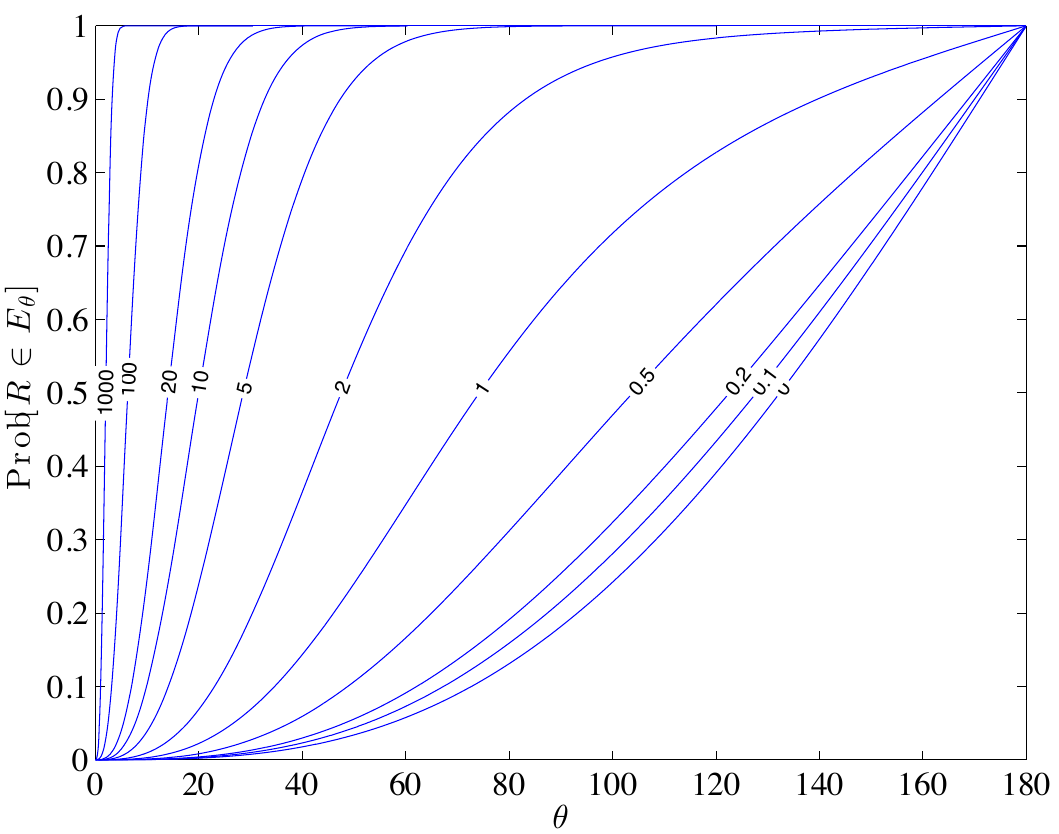}\label{fig:MFsyma}}
	}
\centerline{
	\subfigure[Contour plot of $\mathrm{Prob}{[} R\in E_\theta {]}$]{\includegraphics[width=0.8\columnwidth]{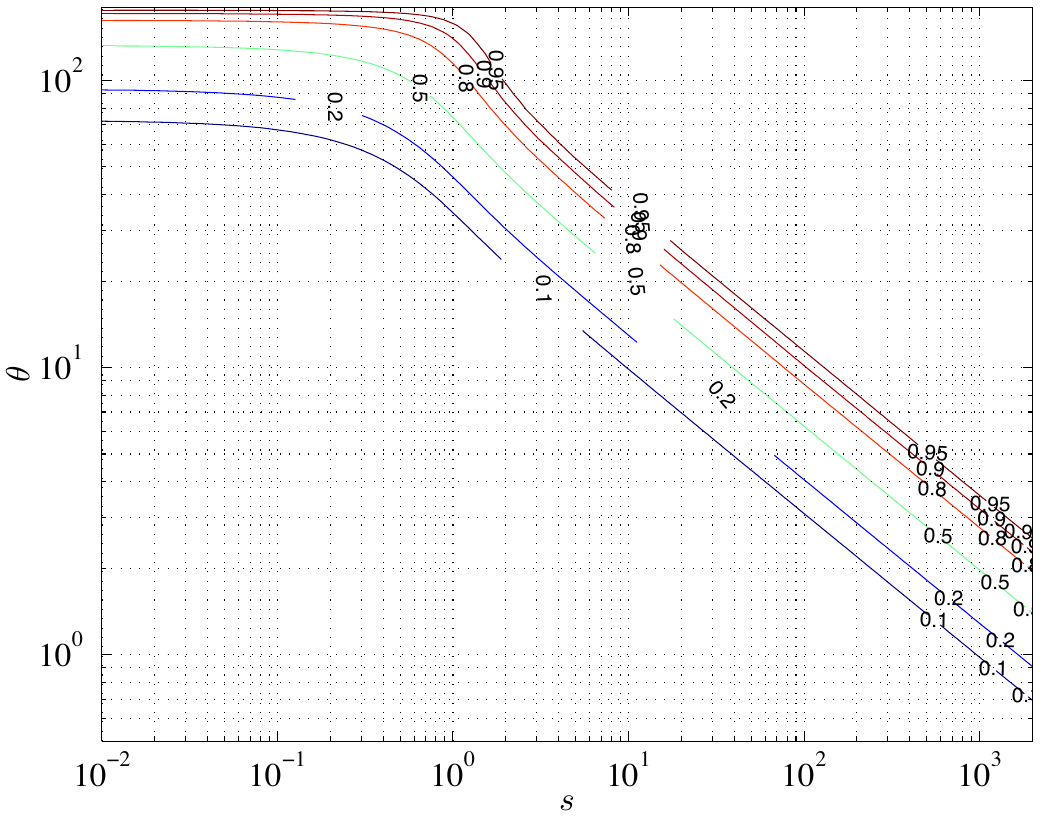}\label{fig:MFsymb}}
	}
\caption{Cumulative distribution for $\mathcal{M}(sI_{3\times 3})$}\label{fig:MFsym}	
\end{figure}	

Since $E_0=\emptyset$ and $E_\pi=\SO$, the probability $\mathrm{Prob}[R\in E_\theta]$ strictly increases from $0$ to $1$ as $\theta$ is varied form $0$ to $\pi$, and it can be considered as the cumulative distribution function for $\mathcal{M}(sI_{3\times 3})$. This is presented in Fig. \ref{fig:MFsyma} for varying $s$. The matrix Fisher distribution is more concentrated to the mean attitude as $s$ increases, and it becomes a uniform distribution on $\SO$ when $s=0$.  By evaluating \refeqn{PrEtheta} repeatedly, one may determine the degree of dispersion quantitatively from $s$, as illustrated in Fig. \ref{fig:MFsymb}. For example, when $R\sim \mathcal{M}(100 I_{3\times 3})$, the angle between $R$ and $I_{3\times 3}$ is less than $10^\circ$ with the probability of $0.9$.

\subsection{Visualization}

While the preceding analysis for the mean attitude, moments, and cumulative distribution characterize the stochastic properties of the matrix Fisher distribution quantitatively, it is still desirable to visualize it for direct, intuitive understanding. However, it is challenging to visualize the probability density \refeqn{MF}, as its domain, $\SO$ is three-dimensional.

A method to visualize a probability density function on $\SO$ has been presented in~\cite{LeeLeoPICDC08}. It is based on the fact that the $i$-th column of the rotation matrix, namely $Re_i$ corresponds to the direction of the $i$-th body-fixed frame resolved in the inertial frame, and it has the unit-length, i.e., $Re_i\in\Sph^2$ for any $i\in\{1,2,3\}$. Therefore, one can find the marginal density for each $Re_i$, and visualize it on the surface of the unit-sphere via color-shading. 

The computation of the marginal density for each column of a rotation matrix would require a double-integration in general. Here, we show that the marginal density of the matrix Fisher distribution can be obtained directly as follows.
\begin{theorem}
Suppose $R\sim\mathcal{M}(F)$ for a matrix parameter $F\in\Re^{3\times 3}$. For $(i,j,k)\in\mathcal{I}$, let $r_i\in\Sph^2$ and $f_i\in\Re^3$ be the $i$-th column of $R$ and $F$, respectively, i.e., $r_i=Re_i$, $f_i=Fe_i$. Also, let $f_{j,k}\in\Re^{3\times 2}$ be composed of the $j$-th column and the $k$-th column of $F$, i.e., $f_{j,k}=[f_j, f_k]$. The marginal density for $r_i$ is given by
\begin{align}
p_i(r_i) = \frac{\exp (f_i^T r_i)}{c(F)}I_0(s'_1+\mathrm{sgn}[r_i^T(f_j\times f_k)]s'_2),\label{eqn:pri}
\end{align}
relative to the uniform distribution on $\Sph^2$, where for $l\in\{1,2\}$,
\begin{align}
s'_l= \sqrt{\lambda_l\!\bracket{f_{j,k}^T (I_{3\times 3}-r_ir_i^T) f_{j,k}}},\label{eqn:sl}
\end{align}
and $\lambda_l[\cdot]$ denotes the $l$-th eigenvalue. 
\end{theorem}
\begin{proof}
For simplicity, we show \refeqn{pri} when $(i,j,k)=(1,2,3)$. For given $r_1\in\Sph^2$, define $\bar r_2,\bar r_3\in\Sph^2$ such that they span the orthogonal complement of $r_1$, which is the two-dimensional plane that is normal to $r_1$. Further assume that they are ordered properly such that the determinant of the matrix $[r_1,\bar r_2, \bar r_3 ]\in\Re^{3\times 3}$ becomes one. As such, the matrix $[r_1,\bar r_2, \bar r_3 ]\in\SO$ is one particular rotation matrix whose first column is equal to $r_1$. The set of \textit{all} rotation matrices whose first column is equal to $r_1$ is obtained by rotating $[r_1,\bar r_2, \bar r_3 ]$ about $r_1$ as
\begin{align*}
\{ R\in\SO\,|\, Re_1=r_1\}=\{[r_1,\bar r_{2,3}Z]\in\SO \,|\, Z\in\mathsf{SO(2)}\},
\end{align*}
where $\bar r_{2,3}=[\bar r_2,\bar r_3]\in\Re^{3\times 2}$, and $\mathsf{SO(2)}=\{Z\in\Re^{2\times 2}\,|\, Z^TZ=I_{2\times 2},\,\mathrm{det}[Z]=1\}$.

Therefore, the marginal density for the first column of $R$ is constructed by integrating \refeqn{MF} over $\mathsf{SO(2)}$ as follows.
\begin{align}
p_1(r_1)&= \frac{1}{c(S)}\int_{Z\in\mathsf{SO(2)}} \exp(\trs{F^T [r_1, \bar r_{2,3}Z]})\, dZ\nonumber\\
&=\frac{\exp(f_1^T r_1)}{c(S)}\int_{Z\in\mathsf{SO(2)}} \exp(\trs{f_{2,3}^T \bar r_{2,3}Z}) \, dZ,\label{eqn:pr1}
\end{align}
where $dZ$ is the Haar measure on $\mathsf{SO(2)}$ that is scaled such that $\int_{\mathsf{SO(2)}}dZ=1$.

Next, we evaluate the integral of \refeqn{pr1}. The singular value decomposition of $f_{2,3}^T \bar r_{2,3}\in\Re^{2\times 2}$ is given by
\begin{align}
f_{2,3}^T\bar r_{2,3} & = U' S' (V')^T,\label{eqn:fr}
\end{align}
where $U',V'\in\Re^{2\times 2}$ with $(U')^T U'= (V')^T V'= I_{2\times 2}$, and $S'=\mathrm{diag}[s_1',s_2']$ with $s_1'\geq s_2'\geq 0$.  Similar to \refeqn{USVp}, the proper singular value decomposition is defined as
\begin{align}
f_{2,3}^T\bar r_{2,3} & = U S V^T,\label{eqn:frp}
\end{align}
where $U,V\in\mathsf{SO(2)}$ and a diagonal matrix $S\in\Re^{2\times 2}$ are given by
\begin{align}
U & = U' \mathrm{diag}[1,\,\mathrm{det}[U']],\\
S & = \mathrm{diag}[s_1',\,\mathrm{det}[U'V']s_2'],\\
V & = V' \mathrm{diag}[1,\,\mathrm{det}[V']].
\end{align}

%

Using this, we change the integration variable of \refeqn{pr1} from $Z$ to $Y=V^T Z U\in\mathsf{SO(2)}$ to obtain
\begin{align*}
\int_{Z\in\mathsf{SO(2)}} \exp(\trs{f_{2,3}^T \bar r_{2,3}Z}) \, dZ=\int_{Y\in\mathsf{SO(2)}} \exp(\trs{SY}) \, dY,
\end{align*}
which is equivalent to the normalizing constant for the matrix Fisher distribution defined on $\mathsf{SO(2)}$. To evaluate it, $\mathsf{SO(2)}$ is parameterized with $\theta\in[0,2\pi)$ via
\begin{align*}
Y(\theta)=\begin{bmatrix} \cos\theta & -\sin\theta \\ \sin\theta & \cos\theta\end{bmatrix}.
\end{align*}
Then, $dY(\theta) = \frac{1}{2\pi} d\theta$, and $\trs{SY(\theta)}= \trs{S} \cos\theta$. Therefore,
\begin{align*}
\int_{Z\in\mathsf{SO(2)}} \exp(\trs{SY}) \, dZ=\frac{1}{2\pi}\int_0^{2\pi} \exp(\trs{S}\cos\theta)\,d\theta.
\end{align*}
From \refeqn{I0} and \refeqn{frp}, this is equivalent to
\begin{align*}
I_0(\trs{S}) = I_0 ( s_1' + \mathrm{det}[U'V']s_2').
\end{align*}
Substituting these into \refeqn{pr1},
\begin{align}
p_1(r_1)
&=\frac{\exp(f_1^T r_1)}{c(S)}I_0(s_1' + \mathrm{det}[U'V']s_2').\label{eqn:pr1_2}
\end{align}

Finally, we find an alternative expression of $\mathrm{det}[U'V']$ in terms of $r_1$ and $F$. First, we consider the case when $\mathrm{det}[S'] >0$. Since $\mathrm{det}[S'] >0$ and $\mathrm{det}[U'V']=\pm 1$,
\begin{align*}
\mathrm{det}[U'V'] & = \mathrm{sgn}[\mathrm{det}[U'V']]=
\mathrm{sgn}[\mathrm{det}[S']\mathrm{det}[U'V']]\\
&=\mathrm{sgn}[\mathrm{det}[f_{2,3}^T \bar r_{2,3}]],
\end{align*}
where the last equality is from \refeqn{fr}. Since $r_1^T \bar r_{2,3}=0_{1\times 2}$,
\begin{align*}
f_{2,3}^T \bar r_{2,3} & = f_{2,3}^T (I_{3\times 3} -r_1r_1^T)\bar r_{2,3}=((I_{3\times 3} -r_1r_1^T)f_{2,3})^T \bar r_{2,3}\\ & \triangleq [\bar f_{2}, \bar f_3]^T \bar r_{2,3},
\end{align*}
where $\bar f_2,\bar f_3\in\Re^3$ coincide with the orthogonal projection of $f_2$ and $f_3$ to the plane normal to $r_1$, respectively, i.e., $\bar f_2^T r_1 = \bar f_3^T r_1=0$, which implies
\begin{align*}
[r_1, \bar f_2, \bar f_3]^T [r_1, \bar r_2, \bar r_3] =  \mathrm{diag}\{1, \bar f_{2,3}^T \bar r_{2,3}\}.
\end{align*}
Therefore,
\begin{align*}
\mathrm{det}[\bar f_{2,3}^T \bar r_{2,3}]&=\mathrm{det}{[r_1, \bar f_2, \bar f_3]}\mathrm{det}{[r_1, \bar r_2, \bar r_3]}\\
& = \mathrm{det}{[r_1, \bar f_2, \bar f_3]}=r_1^T (\bar f_2\times \bar f_3),
\end{align*}
as $\mathrm{det}[r_1, \bar r_2, \bar r_3]=1$ by the construction of $\bar r_2,\bar r_3$. Substituting the definition of $\bar f_2$, and $\bar f_3$, namely  $\bar f_2 = (I_{3\times 3}-r_1r_1^T)f_2$ and $\bar f_3 = (I_{3\times 3}-r_1r_1^T)f_3$,
\begin{align*}
r_1^T &(\bar f_2\times \bar f_3)=\bar f_2^T (r_1\times \bar f_3)\\
& = f_2^T (I_{3\times 3}-r_1r_1^T) \{r_1\times ((I_{3\times 3}-r_1r_1^T) f_3)\},
\end{align*}
which is simplified into $f_2^T (r_1\times f_3)=r_1^T (f_2\times f_3)$.

Summarizing the results derived after \refeqn{pr1_2}, 
\begin{align*}
\mathrm{det}[U'V'] 
&=\mathrm{sgn}[\mathrm{det}[f_{2,3}^T \bar r_{2,3}]]
=\mathrm{sgn}[\mathrm{det}[\bar f_{2,3}^T \bar r_{2,3}]]\\
& = \mathrm{sgn}[r_1^T (f_2\times f_3)].
\end{align*}
Substituting this into \refeqn{pr1_2} yield \refeqn{pri}  when $\mathrm{det}[S']>0$. On the other hand, when $\mathrm{det}[S']=s_1's_2'=0$, we have $s_2'=0$ as $s_1'\geq s_2' \geq 0$. Therefore,  \refeqn{pr1_2} reduces to $p_1(r_1)=\frac{\exp(f_1^T r_1)}{c(S)}I_0(s_1')$ which is equivalent to \refeqn{pri} as $s_2'=0$.

Next, since $s_1',s'_2$ are the singular values of $f_{2,3}^T \bar r_{2,3}$, they are the square-root of the eigenvalues of  $f_{2,3}^T \bar r_{2,3}(f_{2,3}^T \bar r_{2,3})^T=f_{2,3}^T(I_{3\times 3}-r_1r_1^T) f_{2,3}$, which shows \refeqn{sl}. 

These show \refeqn{pri} and \refeqn{sl} for $(i,j,k)=(1,2,3)$, and the cases for other $(i,j,k)\in\mathcal{I}$ can be shown similarly. 
\end{proof}

\begin{figure}
\centerline{
	\subfigure[$F_a=5I_{3\times 3}$]{
		\includegraphics[width=0.3\columnwidth]{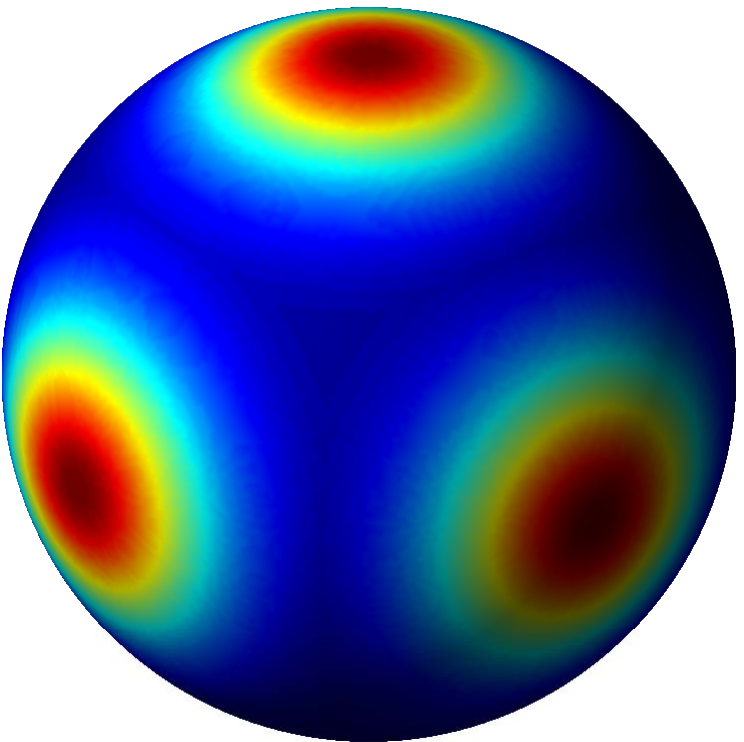}\label{fig:vis1}}
	\hfill
	\subfigure[$F_b=20I_{3\times 3}$]{
		\includegraphics[width=0.3\columnwidth]{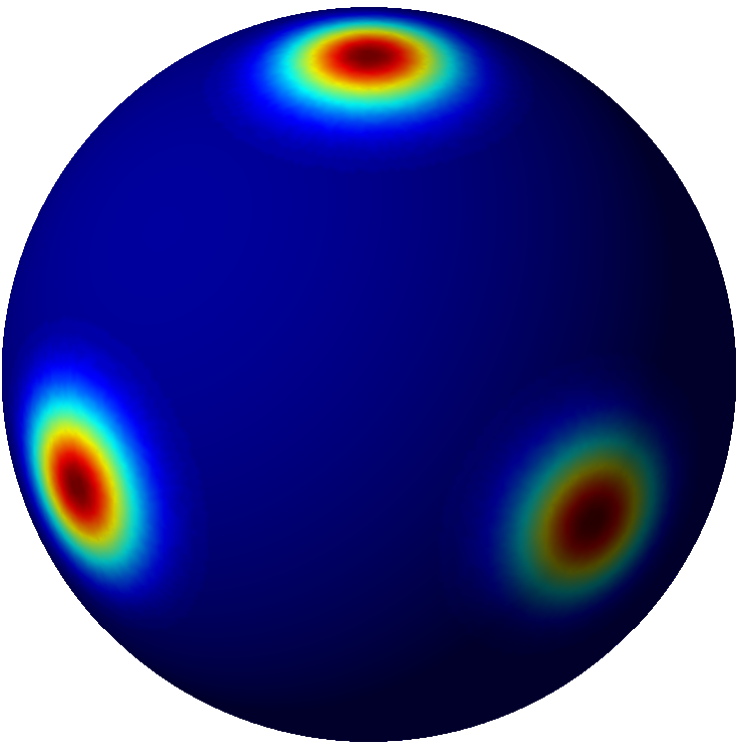}\label{fig:vis2}}
	\hfill
	\subfigure[{$F_c=\mathrm{diag}[25,5,1]$}]{
		\includegraphics[width=0.3\columnwidth]{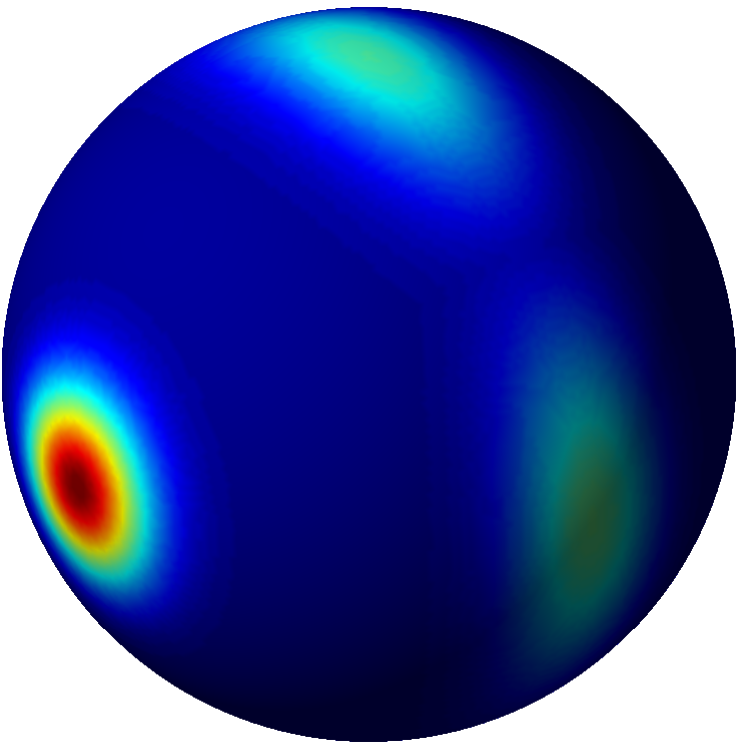}\label{fig:vis3}
		\begin{tikzpicture}[overlay]
		\draw[arrows={-Triangle[angle=30:4pt]}] (-2.9,0.25) -- ++(90:0.5);
		\draw[arrows={-Triangle[angle=30:4pt]}] (-2.9,0.25) -- ++(-30:0.5);
		\draw[arrows={-Triangle[angle=30:4pt]}] (-2.9,0.25) -- ++(210:0.5);
		\node at (-3.3,0.2) {\scriptsize $e_1$};
		\node at (-2.5,0.2) {\scriptsize $e_2$};
		\node at (-2.9,0.85) {\scriptsize $e_3$};
		\end{tikzpicture}
	}
}
\caption{Visualization of selected matrix Fisher distributions: the distribution in (b) is more concentrated than in (a), as the singular values of $F_b$ are greater than those of $F_a$; for both (a) and (b), the distributions of each axis are identical and circular as three singular values of each of $F_a$ and $F_b$ are identical; in (c), the first body-fixed axis (lower left) is more concentrated as the first singular value of $F_c$ is the greatest, and the distributions for the other two axes are elongated. 
}\label{fig:vis}
\end{figure}
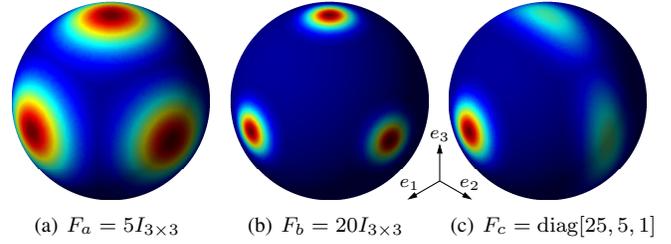

The marginal probability density for each axis of selected matrix Fisher distributions are illustrated in Fig. \ref{fig:vis}.

\subsection{Geometric Interpretation}\label{sec:GI}

Next, we provide the geometric interpretation of the matrix parameter $F$ in determining the shape of the matrix Fisher distribution. Consider a set of rotation matrices parameterized by $\theta_i\in[0,2\pi)$ for $i\in\{1,2,3\}$ as
\begin{align}
R_i(\theta_i) = \exp(\theta_i\widehat {Ue_i}) UV^T=U\exp(\theta_i\hat e_i) V^T.\label{eqn:Ri}
\end{align}
This corresponds to the rotation of the mean attitude $UV^T$ of $\mathcal{M}(F)$, about the axis $Ue_i$ by the angle $\theta_i$, where $Ue_i$ is considered expressed with respect to the inertial frame. 

Using \refeqn{USV}, the probability density \refeqn{MF} along \refeqn{Ri} is given by
\begin{align}
p(R_i(\theta_i)) 
&=\frac{1}{c(S)}\exp(\trs{S\exp(\theta_i\hat{e}_i)}).\label{eqn:pRi0}
\end{align} 
Substituting Rodrigues' formula~\cite{ShuJAS93}, namely $\exp(\theta_i\hat{e}_i)=I_{3\times 3}+\sin\theta_i\hat e_i +(1-\cos\theta_i)\hat e_i^2$, and rearranging, this reduces to
\begin{align}
p(R_i(\theta_i)) 
&=\frac{e^{s_i}}{c(S)}\exp((s_j+s_k)\cos\theta_i),\label{eqn:pRi}
\end{align} 
where $j,k$ are determined such that $(i,j,k)\in\mathcal{I}$. This resembles the von Mises distribution on a circle~\cite{MarJup99}, where the probability density is proportional to $\exp^{\kappa\theta}$ with a concentration parameter $\kappa=s_j+s_k$. As such, when considered as a function of $\theta_i$, the distribution of $p(R_i(\theta_i)) $ becomes a uniform distribution as $s_j+s_k\rightarrow 0$, and it is more concentrated as $s_j+s_k$ increases. It has been also shown that when $s_j+s_k$ is sufficiently large, the von Mises distribution of $\theta_i$ is well approximated by the Gaussian distribution with the variance of $\frac{1}{s_j+s_k}$~\cite{MarJup99}. 

Another noticeable property of \refeqn{pRi} is that the probability density depends only on the singular values $s_i,s_j,s_k$ and the rotation angle $\theta_i$, and it is independent of $U$ or $V$.  Recall \refeqn{Ri} corresponds to the rotation of  the mean attitude $UV^T$ about the $i$-th column of $U$. Consequently, each column of $U$ is considered as the \textit{principle axis} of rotation for $\mathcal{M}(F)$, analogous to the principal axes of a multivariate Gaussian distribution. 

In summary, the role of $F=USV^T$ in determining the shape of the distribution of $\mathcal{M}(F)$ is as follows: (i) the mean attitude is given by $UV^T$; (ii) the columns of the rotation matrix $U$ specify the principle axes of rotations in the inertial frame; (iii) the proper singular vales $S$ describe the concentration of the distribution along the rotations about the principle axes, and in particular, the dispersion along the rotation of the mean attitude about the axis $Ue_i$ is determined by $s_j+s_k$ for $(i,j,k)\in\mathcal{I}$.

\begin{figure}
\centerline{
	\subfigure[{$F_c'=AF_c=(AU)SV^T$}]{\hspace*{0.4cm}
		\includegraphics[width=0.3\columnwidth]{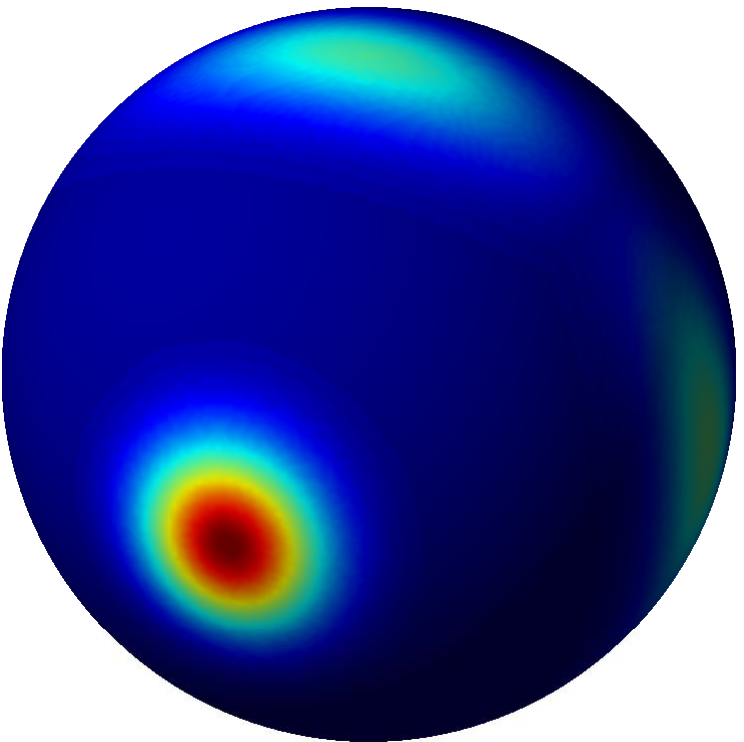}\label{fig:visrot2}
		\begin{tikzpicture}[overlay]
		\draw[arrows={-Triangle[angle=30:4pt]}] (-2.9,0.25) -- ++(90:0.5);
		\draw[arrows={-Triangle[angle=30:4pt]}] (-2.9,0.25) -- ++(-30:0.5);
		\draw[arrows={-Triangle[angle=30:4pt]}] (-2.9,0.25) -- ++(210:0.5);
		\node at (-3.3,0.2) {\scriptsize $e_1$};
		\node at (-2.5,0.2) {\scriptsize $e_2$};
		\node at (-2.9,0.85) {\scriptsize $e_3$};
		\end{tikzpicture}
		\hspace*{0.4cm}}
	\hspace*{0.1cm}
	\subfigure[{$F_c''=(AU)S(U^TA^TUV^T)$}]{\hspace*{0.6cm}
		\includegraphics[width=0.3\columnwidth]{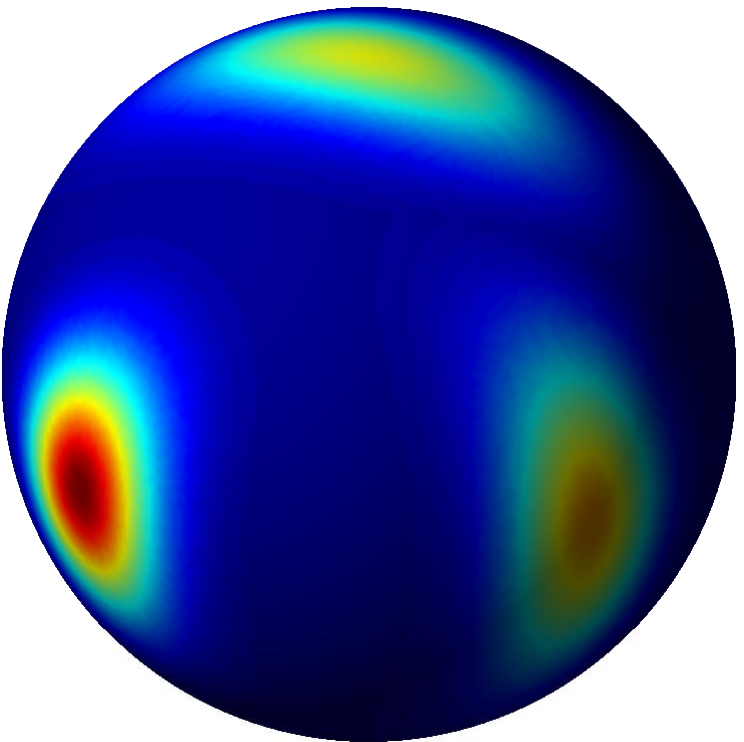}\label{fig:visrot3}\hspace*{0.6cm}}
}
\caption{The role of $U$ and $V$ in determining the shape of the distribution with $F_c=\mathrm{diag}[25,5,1]=USV^T$, and $A=\exp(\frac{\pi}{6}\hat e_3)$: (a) by left-multiplying the matrix parameter $F_c$ with $A$, both the mean attitude and the principal axes are rotated about the third inertial axis pointing upward; (b) the principal axes are identical to (a), but the mean attitude $(AU)(U^TA^TUV^T)=UV^T$ is rotated back to the mean attitude $UV^T$ of $F_c$.}\label{fig:visrot}
\end{figure}
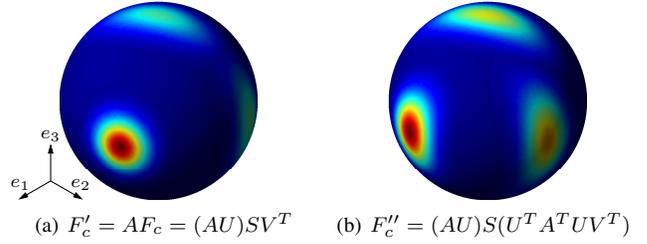

For instance, consider $F_c=\mathrm{diag}[25,5,1]$, where $(s_1,s_2,s_3)=(25,5,1)$ and $U=V=I_{3\times 3}$ (see Figure \ref{fig:visrot}). The mean attitude is $I_{3\times 3}$, and the principal axes are $(e_1,e_2,e_3)$. Since $s_2+s_3=6 \leq  s_3+s_1=26 \leq s_1+s_2=30$, rotating the mean attitude about the first principal axis $e_1$ (lower left) is most dispersed, thereby making the marginal distribution of the second and the third body-fixed axes elongated along the great circle perpendicular to the first principal axis. 

The shape of the distribution can be altered as desired by adjusting the matrices $U$ and $V$ properly. Suppose we wish to purely rotate the distribution without changing the shape of the distribution relative to the mean attitude. In other words, we have to rotate the mean attitude and the direction of the principal axes simultaneously. This is achieved by changing $U$, or equivalently, \textit{left}-multiplying the matrix parameter $F_c$ with a certain rotation matrix. For example, $\mathcal{M}(\exp(\frac{\pi}{6}\hat e_3)F_c)$, corresponds to the matrix Fisher distribution obtained by rotating $\mathcal{M}(F_c)$ about the third inertial axis $e_3$ by $\frac{\pi}{6}$ (see Figure \ref{fig:visrot2}). Also, one may rotate the principal axes without altering the mean attitude by changing both of $U$ and $V$ as illustrated in Figure \ref{fig:visrot3}.

\subsection{Maximum Likelihood Estimation}\label{sec:MLE}

Suppose there is a set of sample attitudes $R_1,R_2,\ldots,R_N$ of $N$ independent and identically distributed observations from a matrix Fisher distribution $\mathcal{M}(F)$, where $F$ is an unknown matrix parameter. We wish to determine the matrix parameter from the samples. 

Define the log-likelihood function as
\begin{align*}
\mathcal{L}(F|R)= -\log c(F) + \trs{F^TR}.
\end{align*}
The joint log-likelihood function is
\begin{align}
\mathcal{L}(F|R_1,\ldots, R_N) &= \sum_{i=1}^N \mathcal{L}(F|R_i).\label{eqn:JLL}
\end{align}
The maximum log-likelihood estimate of $F$ that maximizes \refeqn{JLL} is obtained as follows.

\begin{theorem}\label{thm:MLE}
Let $R_1,\ldots, R_N$ be independent and identically distributed samples from $\mathcal{M}(F)$. The arithmetic mean of the sample attitudes and its proper singular value decomposition are given by
\begin{align}
\bar R = \frac{1}{N}\sum_{i=1}^n R_i= U D V^T,\label{eqn:UDV}
\end{align}
where $U,V\in\SO$ and $D=\mathrm{diag}[d_1,d_2,d_3]\in\Re^{3\times 3}$. 

Suppose that the singular values of $F$ are distinct. The maximum log-likelihood estimate of the matrix parameter is given by
\begin{align}
F = U S V^T,\label{eqn:MLEF}
\end{align}
where $S=\mathrm{diag}[s_1,s_2,s_3]\in\Re^{3\times 3}$, and $(s_1,s_2,s_3)$ is the solution of 
\begin{align}
\frac{1}{c(S)}\deriv{c(S)}{s_i}-d_i =0,\quad \text{for $i\in\{1,2,3\}$}.\label{eqn:MLEs}
\end{align}
\end{theorem}
\begin{proof}
Let $F=USV^T$ be the proper singular value decomposition of $F$. Then, determining $F$ is equivalent to finding $U,V\in\SO$ and the diagonal matrix $S$. We show the matrices $U,S,V$ that maximizes $\mathcal{L}$ is obtained by \refeqn{UDV} and \refeqn{MLEs}. The joint likelihood function is 
\begin{align*}
\frac{1}{N}\mathcal{L}(F|R_1,\ldots, R_N) &= -\log c(S) + \trs{V S U^T \bar R}.
\end{align*}
Since $U,V\in\SO$, their variations can be written as $\delta U = U\hat \zeta_u$ and $\delta V= V\hat \zeta_v$ for $\zeta_u,\zeta_v\in\Re^3$. Using these, the variation of $\mathcal{L}$ is given by
\begin{align}
\frac{1}{N}\delta\mathcal{L}
&= -\deriv{\log c(S)}{s} \delta s
+ \trs{V\hat\zeta_v SU^T \bar R}\nonumber\\
&\quad + \trs{V\mathrm{diag}[\delta s] U^T \bar R}
- \trs{VS\hat\zeta_u U^T \bar R},\label{eqn:dL}
\end{align}
which is zero for any $\delta s,\zeta_u,\zeta_v\in\Re^3$ at the maximum likelihood estimate. This implies that $SU^T \bar R V$ and $ U^T \bar RVS$ are symmetric. Therefore, 
\begin{align*}
SD = D^T S, \text{ and } DS=SD^T,
\end{align*}
where $D=U^T\bar R V$. Combining these,
\begin{align*}
D^T = S D S^{-1}= S^{-1} D S,
\end{align*}
or equivalently $\frac{s_i}{s_j} D_{ij} = \frac{s_j}{s_i} D_{ij}$ for $i\neq j\in\{1,2,3\}$. As $(s_1,s_2,s_3)$ are distinct, it follows $D_{ij}=0$ when $i\neq j$.
Since $D=U^T \bar R V$ is diagonal and $U,V\in\SO$, the proper singular value decomposition of $\bar R$ is given by \refeqn{UDV}. Substituting \refeqn{UDV} into \refeqn{dL}, we obtain \refeqn{MLEs}.
\end{proof}



\section{Global First-Order Attitude Estimation}\label{sec:FAE}

In this section, an attitude estimation scheme is proposed based on the matrix Fisher distribution on $\SO$. Assuming that the initial attitude estimate and the attitude measurement errors are described by matrix Fisher distributions on $\SO$, we construct an estimated attitude distribution via another matrix Fisher distribution according to the Bayesian framework. Therefore, this approach is an example of so-called, \textit{assumed density filters}. 

One issue of any assumed density filter is that the propagated uncertainty is not guaranteed to be distributed as the selected density model. 
This section presents an attitude estimator where  the estimated matrix Fisher distribution is chosen such that its first moment matches with that of the propagated distribution. It is also shown that the matrix Fisher distribution is closed under conditioning with certain types of attitude measurements. 


\subsection{Attitude Estimation Problem Formulation}

Consider a stochastic differential equation on $\SO$,
\begin{align}
(R^T dR)^\vee = \Omega dt + H dW,\label{eqn:SDE}
\end{align}
where $\Omega\in\Re^3$ is the angular velocity resolved in the body-fixed frame, $H\in\Re^{3\times 3}$ is a diagonal matrix, and $W\in\Re^3$ denotes an array of independent, identically distributed Wiener processes. This is defined according to the Ito sense~\cite{KarShr91,Oks14}. It is assumed that $\Omega$ is available from an angular velocity sensor. The Wiener processes scaled by the matrix $H$ correspond to the noise from the angular velocity sensor.

Suppose that the initial attitude follows a matrix Fisher distribution with a known matrix parameter, and the attitude is measured repeatedly. We wish to determine the current estimate of the attitude based on the initial distribution and the measurement history according to the Bayesian framework. The time variable $t$ is discretized with a fixed step size $h>0$, and let the value of a variable at the $k$-th time step be denoted by the subscript $k$.

\subsection{Prediction}

We first show that an analytic expression can be obtained for the first moment propagated along \refeqn{SDE}.

\begin{theorem}\label{thm:PDT}
Suppose $R_k\sim\mathcal{M}(F_k)$ for a given $F_k\in\Re^{3\times 3}$, and assume $\Omega(t)$ is fixed over $t\in[t_k,t_{k+1}]$. Let $R_{k+1}$ be the solution of \refeqn{SDE} at $t_{k+1}$. The first moment of $R_{k+1}$ is given by
\begin{align}
\mathrm{E}[R_{k+1}]&= \mathrm{E}[R_k]\braces{I_{3\times 3}+\frac{h}{2}(-\trs{G_k}I_{3\times 3} + G_k)}\exp(h\hat\Omega_k)\nonumber\\
&\quad +\mathcal{O}(h^{1.5}),\label{eqn:ERkp}
\end{align}
where $G_k=H_kH_k^T\in\Re^{3\times 3}$.
\end{theorem}
\begin{proof}
Let $Y(t)=R(t)\exp(-(t-t_k)\hat\Omega_k)\in\SO$ for $t\geq t_k$.  Since the second order derivative of $Y$ with $R$ is zero, from Ito's lemma,
\begin{align*}
dY(t) & = dR(t)\exp(-(t-t_k)\hat\Omega_k))\nonumber\\
&\quad - R(t)\hat\Omega_k \exp(-(t-t_k)\hat\Omega_k) dt. 
\end{align*}
Substituting \refeqn{SDE},
\begin{align*}
dY(t)&=R(t) (H(t)dW(t))^\vee \exp(-(t-t_k)\hat\Omega_k)\\
& = Y(t)\{ \exp((t-t_k)\hat\Omega_k) H(t) dW(t)\}^\vee, 
\end{align*}
where we have used the fact that $R\hat x R^T = (Rx)^\vee$ for any $R\in\SO$ and $x\in\Re^3$ to obtain the last equality.

The above stochastic differential equation for $Y(t)$ does not have a drift term, and according to Proposition \ref{prop:ERkp0},
\begin{align*}
\mathrm{E}[Y_{k+1}]=\mathrm{E}[Y_{k}]\braces{I_{3\times 3}+\frac{h}{2}(-\trs{G_k}I_{3\times 3}+ G_k)+\mathcal{O}(h^{1.5})},
\end{align*}
as $\exp((t-t_k)\hat\Omega_k) H(t)=H_k$ when $t=t_k$. This yields \refeqn{ERkp} since $Y(t_k)=R(t_k)$, and $Y(t_{k+1})=R(t_k)\exp(-h\hat\Omega_k)$.
\end{proof}

In short, when $R_k\sim\mathcal{M}(F_k)$, the first moment of $R_{k+1}$ is obtained by \refeqn{ERkp} up to the order of $h^{1.5}$. Then, the distribution of $R_{k+1}$ can be approximated by another matrix Fisher distributions $\mathcal{M}(F_{k+1})$ via Theorem \ref{thm:MLE}, such that the first moment of $\mathcal{M}(F_{k+1})$ matches with the value obtained by \refeqn{ERkp}. More explicitly, the procedure to construct $F_{k+1}\in\Re^{3\times 3}$ is as follows: the right-hand size of \refeqn{ERkp} is decomposed using the proper singular value decomposition to obtain $\mathrm{E}[R_{k+1}]=U_{k+1}D_{k+1}V_{k+1}^T$, and \refeqn{MLEs} is solved for $S_{k+1}$ to obtain $F_{k+1}$ from \refeqn{MLEF} (see the steps 10 through 16 of Table \ref{tab:FAE} for details).

This provides a prediction scheme to construct $\mathcal{M}(F_{k+1})$ from $\mathcal{M}(F_{k})$ along the stochastic differential equation \refeqn{SDE} for given $\Omega_k$ and $H_k$, where the first moment of $\mathcal{M}(F_{k+1})$ matches with that of $R_{k+1}$ up to the order of $h^{1.5}$.

\subsection{Correction}\label{sec:Cor}

Next, we analyze the correction step, or the measurement update step, of Bayesian estimation. Two types of attitude measurements are considered. 

\paragraph{Full Attitude Measurement} Suppose there are $N_Z$ attitude sensors, such as inertial measurement units that provide full three-dimensional attitude measurements. In the absence of the measurement error, the sensor measurement $Z_i\in\SO$ from the $i$-th attitude sensor is given by
\begin{align}
Z_i = R.
\end{align}
It is assumed that the measurement error defined by $R^T Z_i\in\SO$ follows a matrix Fisher distribution with the matrix parameter $F_{Z_i}\in\Re^{3\times 3}$, i.e., $R^T Z_i\sim\mathcal{M}(F_{Z_i})$ with $F_{Z_i}\in\Re^{3\times 3}$ for $i\in\{1,\ldots,N_Z\}$. Equivalently, the probability density of the measurement $Z_i$, conditioned by the true attitude $R$ is given by
\begin{align}
p(Z_i|R) = \frac{1}{c(F_{Z_i})} \exp(\trs{F_{Z_i}^T R^T Z_i}),\label{eqn:pZR}
\end{align}
where the property (ii) of Theorem \ref{thm:C} has been used for $c(RF_{Z_i})=c(F_{Z_i})$. 

The matrix parameter $F_{Z_i}$ determines the stochastic characteristics of the $i$-th attitude sensor completely. For example, if the mean attitude of $\mathcal{M}(F_{Z_i})$ is identity, the attitude sensor is unbiased in the sense that the mean attitude measurement of $p(Z_i|R)$ is the true attitude $R$, and the sensor is more accurate if the proper singular values of $F_{Z_i}$ are greater. 

\paragraph{Direction Measurement} Suppose there are $N_z$ distinctive directions, such as the magnetic north or the direction of the gravity, that are prescribed in the inertial frame. These are measured by a set of direction sensors, such as a magnetometer or accelerometer, in the body-fixed frame. 

Let the $i$-th known direction resolved in the inertial frame be given by $a_i\in\Sph^2$. In the absence of measurement noise, the direction measurement resolved in the body-fixed frame, namely $z_i\in\Sph^2$ is given by
\begin{align}
z_i=R^T a_i.
\end{align}
The stochastic properties of the sensor or the measurement errors are characterized by the following conditioned probability density of $z_i|R$,
\begin{align}
p(z_i| R) =\frac{b_i}{4\pi \sinh b_i}\exp(b_i a_i^TB_i^T Rz_i),\label{eqn:pzR}
\end{align}
where $0<b_i\in\Re$ and $B_i\in\SO$. The above distribution is referred to as the von-Mises Fisher distribution of the unit-vectors on $\Sph^2$~\cite{MarJup99}, and it is defined relative to the uniform distribution on $\Sph^2$. This is a single pole distribution centered at $R^TB_ia_i\in\Sph^2$, and it is more concentrated as $b_i$ becomes larger. Therefore, $b_i$ and $B_i$ specify the concentration and the bias of the direction sensor, respectively. For example, if $B_i=I_{3\times 3}$, then the mean direction of \refeqn{pzR} is $R^TB_ia_i=R^Ta_i$ and the corresponding direction sensor is unbiased. 

With these measurements, the a posteriori attitude probability distribution is obtained as follows. 

\begin{theorem}\label{thm:MU}
Suppose the a priori attitude distribution is given by $\mathcal{M}(F)$ with a matrix parameter $F\in\Re^{3\times 3}$. Consider a set of attitude measurements $\{Z_1,\ldots, Z_{N_Z}\}\in\SO^{N_Z}$ distributed by \refeqn{pZR} for some $(F_{z_1},\ldots, F_{Z_{N_Z}})\in(\Re^{3\times 3})^{N_Z}$, and a set of direction measurements $\{z_1,\ldots, z_{N_z}\}\in(\Sph^2)^{N_z}$ following \refeqn{pzR} for some $\{(b_i,B_i)\}_{i=1}^{N_z}\in(\Re,\SO)^{N_z}$. All of the attitude measurements and direction measurements are mutually independent. Then, the a posteriori distribution for $R|(Z_1,\ldots, Z_{N_Z},z_1,\ldots z_{N_Z})$ conditioned by all of the measurements is also a matrix Fisher distribution. Specifically,
\begin{align}
R&|(Z_1,\ldots, Z_{N_Z},z_1,\ldots z_{N_Z})\nonumber\\
&
 \sim \mathcal{M}(F + \sum_{i=1}^{N_Z} Z_iF_{Z_i}^T +\sum_{j=1}^{N_z} b_j B_j a_j z_j^T).\label{eqn:RZz}
\end{align}
\end{theorem}
\begin{proof}
Let $\mathcal{Z}=(Z_1,\ldots, Z_{N_Z},z_1,\ldots z_{N_Z})$. 
According to Bayes' rule, 
\begin{align*}
p(R|\mathcal{Z})= \frac{p(\mathcal{Z}|R) p(R)}{p(\mathcal{Z})}\propto p(\mathcal{Z}|R) p(R) ,
\end{align*}
as $p(\mathcal{Z})$ does not depend on $R$. As all of the measurements are assumed to be mutually independent,
\begin{align*}
p(R|\mathcal{Z})\propto \braces{\prod_{i=1}^{N_Z} p(Z_i|R)}\braces{\prod_{j=1}^{N_z} p(z_i|R)}p(R).
\end{align*}
Substituting \refeqn{pZR}, \refeqn{pzR}, and \refeqn{MF} and rearranging,
\begin{align*}
p(R|\mathcal{Z})\propto \exp(\trs{(F + \sum_{i=1}^{N_Z} Z_iF_{Z_i}^T +\sum_{j=1}^{N_z} b_j B_j a_j z_j^T)^T R}),
\end{align*}
which shows \refeqn{RZz}.
\end{proof}

Therefore, this may handle a wide variety of sensor configuration, including a single direction measurement, or a combination of attitude sensors and directions sensors with various accuracies and biases. In contrast to the prediction step, no approximation or projection is required, as the a posteriori distribution is guaranteed to be a matrix Fisher distribution. 

\subsection{First-Order Attitude Estimation}

By combining the prediction scheme of Theorem \ref{thm:PDT} and the correction step of Theorem \ref{thm:MU}, a Bayesian attitude estimator based on the matrix Fisher distribution on $\SO$ is formulated, as summarized in Table \ref{tab:FAE}. This is considered  a first-order estimator as the prediction step matches the first moment of the propagated attitude distribution. Note that this should be distinguished from the first moment matching for the Gaussian distributions on $\Re^n$, as there are nine elements in the first moment $\mathrm{E}[R]$ that carry the stochastic information on the mean attitude and the distribution.  

The proposed attitude estimator does not require any assumption on the degree of uncertainties and estimation errors, and it does not rely on any approximation such as linearization. Therefore, it is capable of handling large uncertainties and estimation errors in an intrinsic fashion on the special orthogonal group. There are the unique features of the proposed attitude estimator, compared with the current variations of extended Kalman filter or unscented Kalman filters developed in terms of quaternions.

\newcommand{\algrule}[1][.2pt]{\par\vskip.2\baselineskip\hrule height #1\par\vskip.2\baselineskip}

\begin{table}
\caption{First-Order Attitude Estimation}\label{tab:FAE}
\begin{algorithmic}[1]
\algrule[0.8pt]
\Procedure{First-Order Attitude Estimation}{}
\algrule
	\State $R_0\sim \mathcal{M}(F_0)$, $k=0$
	\Repeat
		\State $F_{k+1}=${\fontshape{sc}\selectfont Propagation}$(F_k, \Omega_k, H_k)$
		\State $k=k+1$
	\Until{$Z_{k+1}$ or $z_{i_{k+1}}$ is available}
	\State $F_{k+1}=${\fontshape{sc}\selectfont Correction}$(F_{k+1}, Z_{k+1},z_{k+1})$
	\State \textbf{go to} Step 3
\EndProcedure
\algrule
\Procedure{$F_{k+1}$=Propagation}{$F_k$, $\Omega_k$, $H_k$}
	\State Compute $\mathrm{E}[R_k]$ with $F_k$ from \refeqn{ER}
	\State Compute $\mathrm{E}[R_{k+1}]$ with $\Omega_k$ and $H_k$ from \refeqn{ERkp}
	\State Perform the singular value decomposition of $\mathrm{E}[R_{k+1}]$ to obtain $U,D,V$ from \refeqn{UDV}
	\State Solve \refeqn{MLEs} for $S$ with $D$
	\State Compute $F_{k+1}$ with $U,S,V$ from \refeqn{MLEF}
\EndProcedure
\algrule
\Procedure{$F^+$=Correction}{$F^-$, $Z$, $z$}
	\State Compute $F^+$ from \refeqn{RZz} with $Z,z$
\EndProcedure
\algrule[0.8pt]
\end{algorithmic}
\end{table}

\section{Global Unscented Attitude Estimation}\label{sec:UAE}

The attitude estimator presented in the previous section based on the averaged solution \refeqn{ERkp} to the stochastic differential equation \refeqn{SDE}. Alternatively, one may approximate a probability distribution on $\SO$ by selected attitudes and weights and propagate them along the solution of the stochastic differential equation, such as in unscented filters~\cite{JulUhlITAC00}. In nonlinear estimation on $\Re^n$, unscented filters are favored as there is no need for sacrificing the relatively well-known information of the dynamic system via linearization, and they have a higher accuracy, compared with the popular extended Kalman filter for nonlinear systems, 

This section provides an unscented transform to approximate a matrix Fisher distribution on $\SO$ via selected attitudes and weighting parameters. Based on this, an unscented attitude estimator is constructed. 

\subsection{Unscented Transform}\label{sec:UT0}

Suppose $R\sim\mathcal{M}(F)$. We wish to define a set of rotation matrices and weights that approximate $\mathcal{M}(F)$.  In the unscented transformation for the Gaussian distribution in $\Re^n$, the sigma points may be chosen along the principle axes of the Gaussian ellipsoid. Motivated by this and the geometric interpretation of the matrix Fisher distribution summarized in Section \ref{sec:GI}, the following unscented transform is proposed for the matrix Fisher distribution on $\SO$.

\begin{definition}{(Unscented Transform)}\label{def:UT}
Consider a matrix Fisher distribution $\mathcal{M}(F)$ for a matrix parameter $F\in\Re^{3\times 3}$. Let the proper singular value decomposition of $F$ is given by \refeqn{USVp}. 

The set of seven sigma points $\Sigma_R(F)\subset\SO^7$ is defined as
\begin{align}
\Sigma_R(F)&=\{R_0=UV^T\}\cup\{\,R_i(\theta_i),R_i(-\theta_i)\,|\, i\in\{1,2,3\}\},\label{eqn:SP}
\end{align}
where $R_i(\theta_i)$ is given by \refeqn{Ri}, with $0<\theta_i<\pi$ defined as
\begin{subnumcases}{\label{eqn:costhetai} \cos\theta_i=}
\sigma + \dfrac{(1-\sigma)(\log c(S) - s_i)}{s_j+s_k}\;  \mbox{if $s_j+s_k\geq 1$},\label{eqn:costhetaia}\\
\{\sigma + (1-\sigma)(\log c(S) - s_i)+\frac{1}{2}\}(s_j+s_k)-\frac{1}{2}\nonumber\\  \mbox{else if $0\leq s_j+s_k < 1$,}\label{eqn:costhetaib}
\end{subnumcases}
for $(i,j,k)\in\mathcal{I}$ with a parameter $\sigma$ satisfying
\begin{align}
0\leq \max\left\{ \frac{2s_1+s_2-s_3-1}{2s_1+s_2-s_3+1} ,\frac{s_1-s_3}{s_1+s_2}\right\} < \sigma < 1.\label{eqn:sigma}
\end{align}

Next, let $w_0\in\Re$ be the weight of $R_0=UV^T$, and let $w_i\in\Re$ be the weight of $R_i(\pm\theta_i)$. The set of weights $\Sigma_w(F)=\{w_0\}\cup\{w_i,w_i\,|\,i\in\{1,2,3\}\}\subset\Re^7$ for each sigma point of \refeqn{SP} is defined by
\begin{align}
w_i & = \frac{1}{4(1-\cos\theta_i)}\braces{\frac{1}{c(S)}\parenth{\deriv{c(S)}{s_i}-\deriv{c(S)}{s_j}-\deriv{c(S)}{s_k}}+1},\label{eqn:Wi}\\
w_0 & = 1 - 2\sum_{i=1}^3 w_i.\label{eqn:W0}
\end{align}
The concatenated map $\Sigma(F)=(\Sigma_R(F),\Sigma_w(F))\subset(\SO\times\Re)^7$ is referred to as the unscented transform of $\mathcal{M}(F)$. 
\end{definition}

In other words, for a given parameter matrix $F$, the seven sigma points are chosen as the mean attitude, and rotations of the mean attitude about each principle axis by the angle determined by \refeqn{costhetai}. Note that each sigma point is guaranteed to be a rotation matrix in $\SO$, and the sum of the weighting parameters is equal to one.

The equation \refeqn{costhetai} to select the rotation angle is motivated as follows. Consider the first case \refeqn{costhetaia} when $s_j+s_k\geq 1$. For any $i\in\{1,2,3\}$, substituting \refeqn{costhetaia} into \refeqn{pRi} and taking logarithm, we obtain
\begin{align*}
\log  p (R_i(\pm \theta_i))
&=\sigma(-\log c(S) + \trs{S}),
\end{align*}
which follows that the last six sigma points of \refeqn{SP} have the identical value of the probability density, given by $\frac{1}{c(S)}\exp (\sigma \trs{S})$. This is similar to the conventional unscented transform for a Gaussian distribution in $\Re^n$, where the non-central sigma points share the same value of the density. The ratio of the shared probability density to the maximum density, $\frac{1}{c(S)}\exp(\trs{S})$ is given by $\exp((\sigma-1)\trs{S})$. Therefore, the last six sigma points will be closer to the mean attitude, when the distribution is concentrated with larger singular values, or when $\sigma$ becomes greater. As $\sigma\rightarrow 1$, all of the sigma points converge to the mean attitude. 

\begin{figure}
\centerline{
	\subfigure[$F_a=5I_{3\times 3}$]{
		\includegraphics[width=0.3\columnwidth]{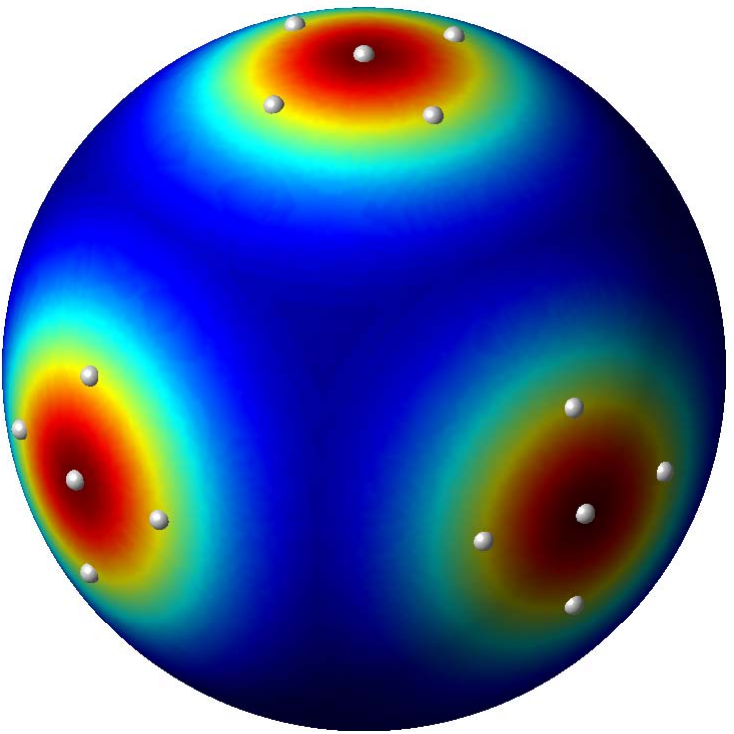}\label{fig:sig1}}
	\hfill
	\subfigure[$F_b=20I_{3\times 3}$]{
		\includegraphics[width=0.3\columnwidth]{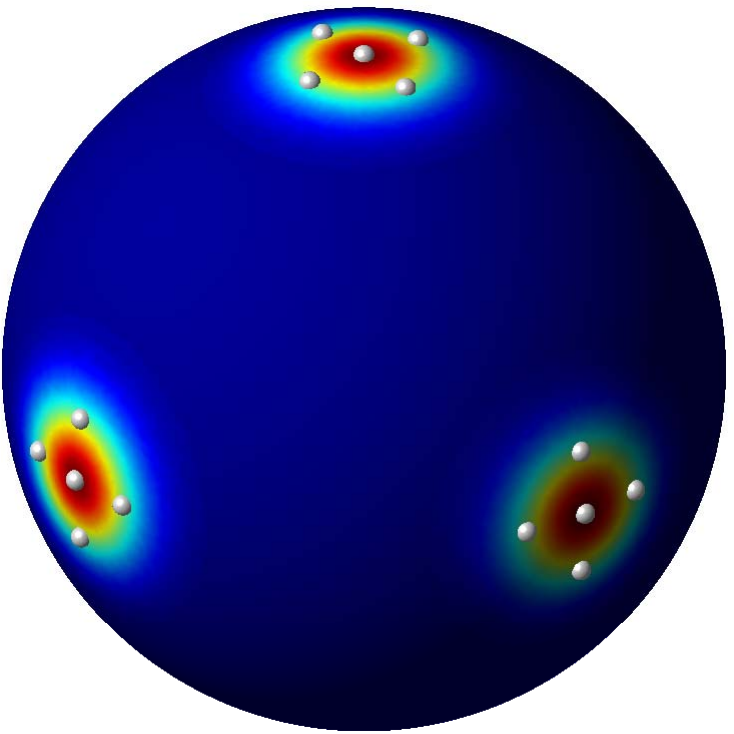}\label{fig:sig2}}
	\hfill
	\subfigure[{$F_c=\mathrm{diag}[25,5,1]$}]{\hspace*{0.005\textwidth}
		\includegraphics[width=0.3\columnwidth]{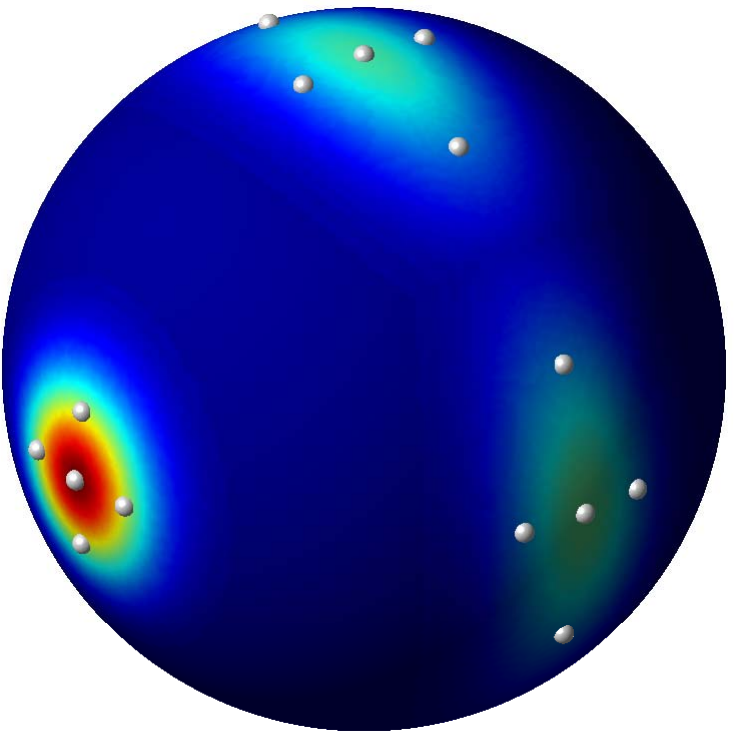}\label{fig:sig3}\hspace*{0.005\textwidth}}
}
\centerline{
	\subfigure[$\cos\theta_i$ with $\sigma=0.9$ and $s_i=15$]{
		\includegraphics[width=0.75\columnwidth]{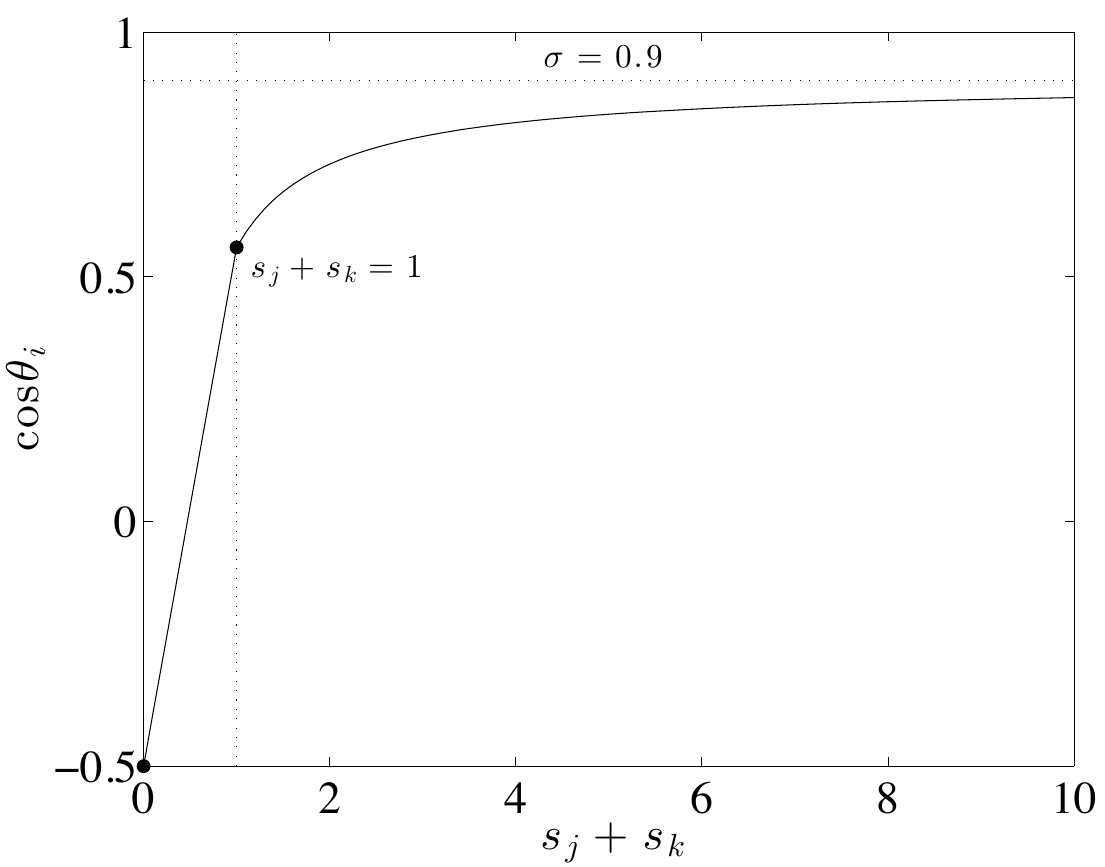}\label{fig:costhetai}}
}
\caption{Visualization of sigma points with $\sigma=0.9$: (a-c) the body-fixed axes of each sigma point for selected matrix Fisher distributions are illustrated; (d) $\cos\theta_i$ defined at \refeqn{costhetai} as a function of $s_j+s_k$}\label{fig:sig}
\end{figure}

However, when $\theta_i$ becomes more uniformly distributed as $s_j+s_k\rightarrow 0$, the first case \refeqn{costhetaia} is ill-conditioned. This motivates the second case \refeqn{costhetaib}.  When $s_j+s_k=0$, \refeqn{costhetaib} yields $\theta_i=\frac{\pi}{3}$, and therefore, there are three uniformly distributed sigma points, namely $R_0=R_i(0)$, $R_i(\frac{\pi}{3})$, and $R_i(-\frac{\pi}{3})$, along the rotation about the $Ue_i$ axis. The equation \refeqn{costhetaib} is constructed such that $\theta_i=\frac{\pi}{3}$ when $s_j+s_k=0$ and it is continuously connected to the first case \refeqn{costhetaia} at $s_j+s_k=1$. 

In short, the rotation angle $\theta_i$ are chosen such that the sigma points become uniformly distributed along the principal axis $Ue_i$ when $s_j+s_k=0$, and they are more concentrated about the mean attitude when $s_j+s_k$ becomes greater. The overall degree of concentration is also controlled by the parameter $\sigma$. Figure \ref{fig:sig} illustrates the  resulting $\cos\theta_i$ as a function of $s_j+s_k$, as well as the sigma points for selected matrix Fisher distributions for a certain selected case.

Next, we show that the first moments of the weighted sigma points $\Sigma(F)$ match with those of the matrix Fisher distribution $\mathcal{M}(F)$.

\begin{theorem}\label{thm:UT}
Consider $R\sim \mathcal{M}(F)$ and the unscented transform $\Sigma(F)$ defined in Definition \ref{def:UT}. The (weighted) first moment of the sigma points $\Sigma(F)$ is equal to that of $\mathcal{M}(F)$, i.e., 
\begin{align}
\sum_{(\tilde R,\tilde w)\in\Sigma(F)} \tilde w \tilde R =\mathrm{E}[R].\label{eqn:M1S0}
\end{align}
\end{theorem}
\begin{proof}
We first show that the presented unscented transform, especially $\theta_i$ is well defined via \refeqn{costhetai}. From the property (iii) of Theorem \ref{thm:C}, and since $-s_1\leq -s_i$,
\begin{align}
 -2 s_1-s_2+s_3 \leq \log c(S) -s_i \leq s_j+s_k,\label{eqn:logcsi}
\end{align}
for any $(i,j,k)\in\mathcal{I}$, where the inequalities become strict when $S\neq 0_{3\times 3}$. 

Using this, an upper bound of \refeqn{costhetaia} is given by
\begin{align*}
\sigma + (1-\sigma)\frac{\log c(S)-s_i}{s_j+s_k}< \sigma + (1-\sigma)\frac{s_j+s_k}{s_j+s_k}=  1,
\end{align*}
and its lower bound is 
\begin{align*}
\sigma + (1-\sigma)\frac{\log c(S)-s_i}{s_j+s_k}
& > \sigma + (1-\sigma)\frac{-2s_1-s_2+s_3}{s_2+s_3},
\end{align*}
which is greater than $-1$ as $\frac{s_1-s_3}{s_1+s_2} < \sigma < 1$. Therefore, the right hand side of \refeqn{costhetaia} belongs to $(-1,1)$. 

Similarly, from \refeqn{logcsi}, an upper bound of \refeqn{costhetaib} is
\begin{align*}
\{\sigma & + (1-\sigma)(s_j+s_k)+\frac{1}{2}\}(s_j+s_k)-\frac{1}{2},
\end{align*}
which is strictly less than 1 as $0< \sigma <1$ and $0\leq s_j+s_k< 1$. A lower bound of \refeqn{costhetaib} is 
\begin{align*}
\{\sigma + (1-\sigma)(-2s_1-s_2+s_3)+\frac{1}{2}\}(s_j+s_k)-\frac{1}{2},
\end{align*}
which is a linear function of $\sigma$. When $\sigma =1$, this reduces to $\frac{3}{2}(s_j+s_k)-\frac{1}{2} < 1$, and when $\sigma = \frac{2s_1+s_2-s_3-1}{2s_1+s_2-s_3+1}$, this is simplified into $-\frac{1}{2}(s_j+s_k)-\frac{1}{2} > -1$ as $s_j+s_k< 1$. In short, the absolute value of the right hand sides of \refeqn{costhetai} is strictly less than 1, and therefore $\theta_i$ is well defined. 

Next, the first moment of the sigma points are given by
\begin{align*}
\sum_{(\tilde R,\tilde w)\in\Sigma(F)} \tilde w \tilde R=w_0UV^T + \sum_{i=1}^3 w_i\{R_i(\theta_i)+R_i(-\theta_i)\}.
\end{align*}
We have $R_i(\theta_i)+R_i(-\theta_i)=U\{ \exp(\theta_i\hat e_i)+\exp(-\theta_i\hat e_i)\}\}V^T = 2U\{I_{3\times 3}+(1-\cos\theta_i)\hat e_i^2\}V^T$ from \refeqn{Ri} and Rodrigues' formula. Substituting $\hat e_i^2 = e_ie_i^T - I_{3\times 3}$ and rearranging, the first moment reduces to
\begin{align}
\sum_{(\tilde R,\tilde w)\in\Sigma(F)} \tilde w \tilde R= UDV^T,\label{eqn:M1S}
\end{align}
where $D\in\Re^{3\times 3}$ is a diagonal matrix, whose $i$-th diagonal element for $(i,j,k)\in\mathcal{I}$ is given by
\begin{align*}
D_{ii} = w_0 + 2(w_i + \cos\theta_jw_j + \cos\theta_k w_k). 
\end{align*}
Substituting \refeqn{W0} and \refeqn{Wi}, this reduces to 
\begin{align*}
D_{ii} & = 1 +2(\cos\theta_j-1)w_j +2(\cos\theta_k-1)w_k\\
& = 1 -\frac{1}{2}\braces{\frac{1}{c(S)}\parenth{\deriv{c(S)}{s_j}-\deriv{c(S)}{s_k}-\deriv{c(S)}{s_i}}+1}\nonumber\\
&\quad -\frac{1}{2}\braces{\frac{1}{c(S)}\parenth{\deriv{c(S)}{s_k}-\deriv{c(S)}{s_i}-\deriv{c(S)}{s_j}}+1}\\
& = \frac{1}{c(S)}\deriv{c(S)}{s_i}.
\end{align*}
Therefore, $D=\mathrm{E}[Q]$ according to \refeqn{M1}, and the first moment of the sigma points given at \refeqn{M1S} is equivalent to the first moment of $\mathcal{M}(R)$ at \refeqn{ER}.
\end{proof}

\begin{remark}{(Inverse of Unscented Transform)}\label{rem:UT}
Theorem \ref{thm:UT} yields the inverse of the proposed unscented transform to obtain the matrix Fisher distribution $\mathcal{M}(F)$ for given sigma points and weights $\Sigma(F)$. More explicitly, the first moment $\mathrm{E}[R]$ is computed as a weighted sum as \refeqn{M1S0}. The remaining procedure to obtain $F$ is similar to the maximum log-likelihood estimation summarized in Theorem \ref{thm:MLE}: perform the proper singular value decomposition of $\mathrm{E}[R]$ to obtain $U,D,V$ from \refeqn{UDV}; solve \refeqn{MLEs} for $S$; the matrix parameter $F$ is constructed by $F=USV^T$ from \refeqn{MLEF}.
\end{remark}

\subsection{Unscented Attitude Estimation}

Based on the proposed unscented transform, we construct a Bayesian attitude estimator as follows. In unscented filtering in $\Re^n$, the effects of the process noise is handled either by augmenting the state vector to include the noise, or by adding the covariance of the noise to the propagated covariance of the state. For the stochastic differential equation \refeqn{SDE} considered in this paper, the process noise corresponds to the last term $HdW$. Augmenting the state of \refeqn{SDE} into $(R, W)\in\SO\times\Re^3$ and constructing an unscented transform on $\SO\times\Re^3$ is beyond the scope of this paper.

\begin{table}
\caption{Unscented Attitude Uncertainty Propagation}\label{tab:UAP}
\begin{algorithmic}[1]
\algrule[0.8pt]
\Procedure{$F_{k+1}$=Unscented Propagation}{$F_k$, $\Omega_k$, $H_k$}
\algrule
	\State Compute the set of 7 sigma points and weights $\Sigma(F_k)$ from Definition \ref{def:UT}
	\State Propagate each sigma point with $R_{i_{k+1}}=R_{i_k}\exp(h\hat\Omega_k)$
	\State Compute the first moment of the propagated attitudes $E[R_{k+1}]$ via \refeqn{M1S0}
	\State Update the first moment according to \refeqn{ERkp} with $G_k=hH_kH_k^T$ by \[E[R_{k+1}]=E[R_{k+1}]\braces{I_{3\times 3}+\frac{h}{2}(-\trs{G_k}I_{3\times 3} + G_k)}\] 
	\State Perform the singular value decomposition of $\mathrm{E}[R_{k+1}]$ to obtain $U,D,V$ from \refeqn{UDV}
	\State Solve \refeqn{MLEs} for $S$ with $D$
	\State Compute $F_{k+1}$ with $U,S,V$ from \refeqn{MLEF}
\EndProcedure
\algrule[0.8pt]
\end{algorithmic}
\end{table}

Instead, the effects of $HdW$ are incorporated via \refeqn{ERkp}. More precisely, consider $R_k\sim \mathcal{M}(F_k)$ for a given $F_k\in\Re^{3\times 3}$. The seven sigma points are constructed according to Definition \ref{def:UT}, and each one is propagated along \refeqn{SDE} with $HdW=0$ via $R_{i_{k+1}}=R_{i_k} \exp (h\hat\Omega_k)$. The weighted first moment of the propagated sigma points is obtained by \refeqn{M1S}, and it is multiplied by $I_{3\times 3}+\frac{h}{2}(-\trs{G_k}I_{3\times 3} + G_k)$ to obtain $\mathrm{E}[R_{k+1}]$. The remaining procedure to construct the propagated matrix Fisher distribution $\mathcal{M}(F_{k+1})$ is identical to the steps 13 through 15 of Table \ref{tab:FAE}. This corresponding propagation scheme based on the proposed unscented transform is summarized in Table \ref{tab:UAP}.

An unscented Bayesian attitude estimation scheme is formulated, by replacing the propagation procedure of the first-order attitude estimator introduced in the prior section with the unscented propagation of the current section (more precisely, by substituting the steps 10 through 16 of Table \ref{tab:FAE} with Table \ref{tab:UAP}). There is no need to introduce a new correction procedure with the unscented tranform, as the results of Theorem \ref{thm:MU} is exact and complete.

\section{Numerical Examples}\label{sec:NE}

\begin{figure}
\centerline{
	\subfigure[Angular velocity: true (red), measured (blue) ($\mathrm{rad/s}$)]{
		\includegraphics[width=0.45\columnwidth]{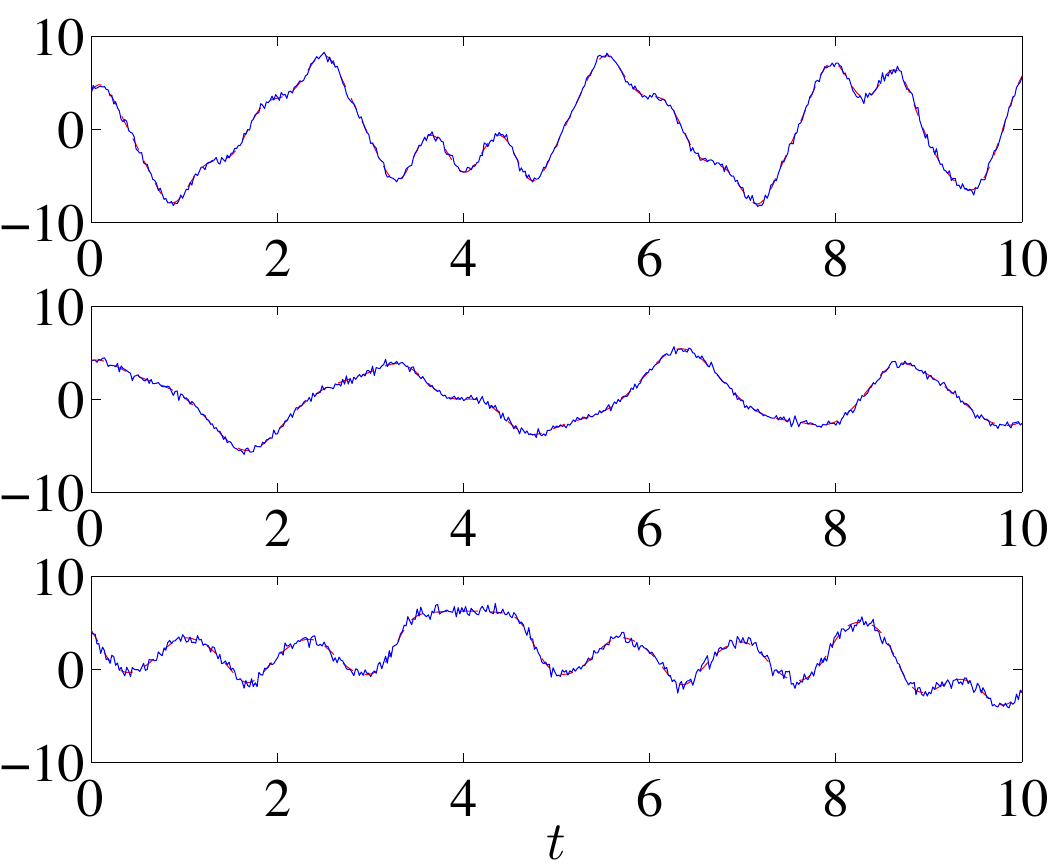}\label{fig:W}}
	\hfill
	\subfigure[Attitude measurement error (deg)]{
		\includegraphics[width=0.45\columnwidth]{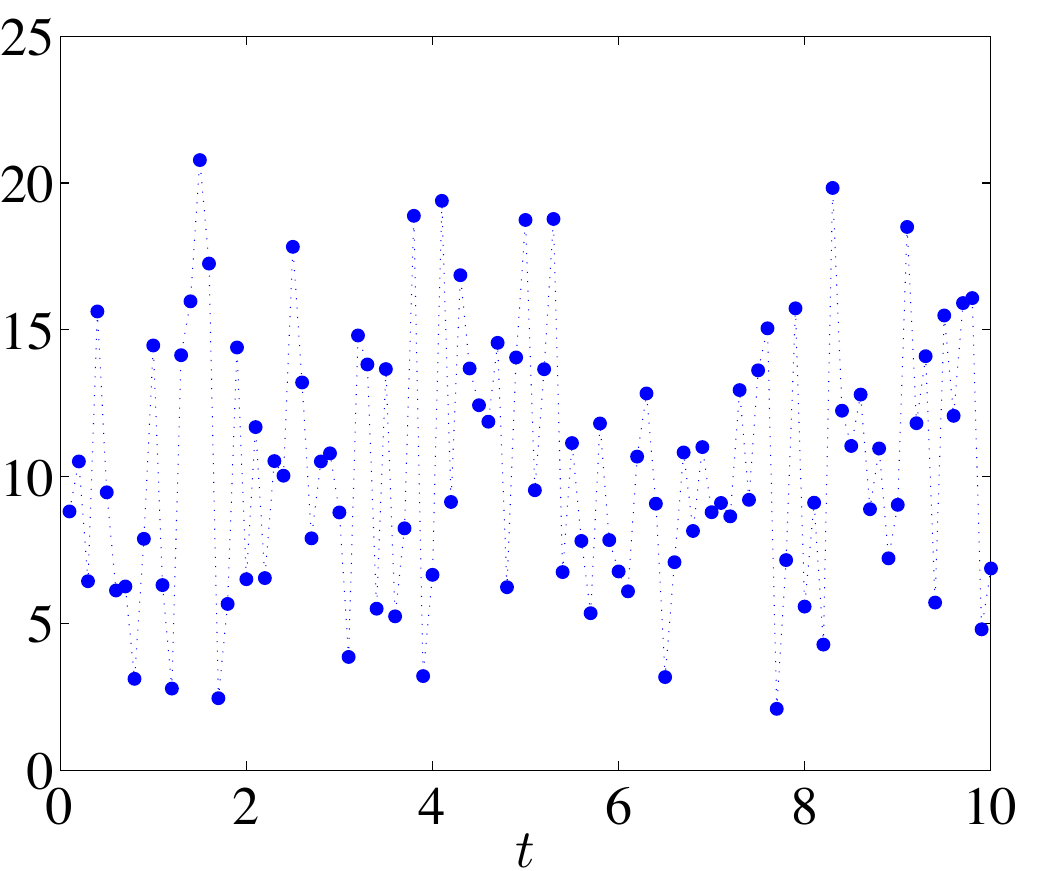}\label{fig:R_mea_err}}
}
\caption{Angular velocity trajectories and measurement errors}\label{fig:err}
\end{figure}

We apply the proposed attitude estimators for a complex attitude dynamics of a rigid body  acting under a uniform gravity, referred to as a 3D pendulum. It is shown that a 3D pendulum may exhibit highly irregular attitude maneuvers, and we adopt a particular nontrivial maneuver presented in~\cite{LeeChaPICDC07} as the true attitude and angular velocity for the numerical example considered in this section. The initial attitude and the angular velocity are $R_{true}(0)=I_{3\times 3}$ and $\Omega_{true}(0)=4.14\times [1,\,1,\,1]\,\mathrm{rad/s}$. 

It is assumed that there is a single full attitude sensor and a gyro, which measure the attitude and the angular velocity at the rates of $10\,\mathrm{Hz}$, and $50\,\mathrm{Hz}$, respectively. The matrix parameter for the attitude measurement error is chosen as $F_Z= \mathrm{diag}[40,50,35]$, and the rotation matrix $Z_i$ representing the attitude measurement error is sampled with $Z_i|R\sim \mathcal{M}(F_Z)$, according to the rejection method described in~\cite{KenGan13}.  The resulting mean attitude measurement error is $10.45^\circ$. The measurement error for the angular velocity is defined by $H=\mathrm{diag}[1.8,\, 1.6,\, 2.4]$, and the mean length of the angular velocity measurement error is $0.45\,\mathrm{rad/s}$. The trajectories of the angular velocity and the measurement errors are illustrated in Figure~\ref{fig:err}.

To implement the proposed estimators, \refeqn{MLEs} should be solved for $S$. Newton's method with a line search is used with the gradient computed by \refeqn{ddcSii} and \refeqn{ddcS}. A software package composed with Matlab is available in~\cite{LeeGit17}.

    
\begin{table}
\caption{Averaged values for steady-state responses after $t=0.5$}\label{tab:EST}
\begin{tabular}{cccccc}\toprule
        &             &  est. err.        & $s_2+s_3$ & $s_3+s_1$ & $s_1+s_2$ \\\midrule
Case I  & First Order &    $6.32^\circ$ & $95.89$ & $119.78$ & $124.88$\\
	    & Unscented   &    $6.32^\circ$ & $96.14$ & $119.70$ & $124.98$\\
		& MEKF~\cite{CraMarJGCD07}		  &  $10.18^\circ$ & N/A & N/A & N/A\\	    
\midrule
Case II & First Order &   $8.70^\circ$ &  $35.74$ &  $37.56$ &  $39.10$\\
	    & Unscented   &    $7.91^\circ$ &  $50.36$ &  $53.54$ &  $54.28$\\\bottomrule
\end{tabular}
\end{table}

\subsection{Case I: Large initial Estimation  Error}

We consider two cases depending on the estimate of the initial attitude. For Case I, the initial matrix parameter is
\begin{align*}
F(0) = 100 \exp (\pi\hat e_1),
\end{align*}
where the initial mean attitude is $\mathrm{M} (0) = \exp (\pi\hat e_1)$, which corresponds to $180^\circ$ rotation of $R_{true}(0)$ about the first body-fixed axis. It is highly concentrated, since $S(0)=100I_{3\times 3}$ is relatively large. As such, this represents the case where the estimator is falsely too confident about the completely incorrect attitude. 

The results of the first-order attitude estimator proposed in Section \ref{sec:FAE} and the unscented estimator presented in Section \ref{sec:UAE}  are summarized in Table \ref{tab:EST} and Figure \ref{fig:ES_0}, where the attitude estimation error is presented, and the degree of uncertainty in the estimates are measured via $\frac{1}{s_i+s_j}$. The results of multiplicative extended Kalman filter (MEKF)~\cite{CraMarJGCD07} are also presented in Table \ref{tab:EST}. 

While the unscented filter exhibits slightly smaller estimation errors and faster convergence, there is no meaningful difference from the first-order estimator. For both estimators, the attitude estimation error rapidly reduces to below $4^\circ$ from the initial error of $180^\circ$ after three attitude measurements at $t=0.3$, and the mean attitude error afterward is about $9^\circ$. The uncertainties in the attitude increase until $t=0.3$ since the measurements strongly conflict with the initial estimate, but they decrease quickly after the attitude estimate converges. Both approaches yield smaller estimation error than MEKF. 

\begin{figure}
\centerline{
	\subfigure[Attitude estimation error ($\mathrm{deg}$)]{\hspace*{0.02\columnwidth}
		\includegraphics[height=0.36\columnwidth]{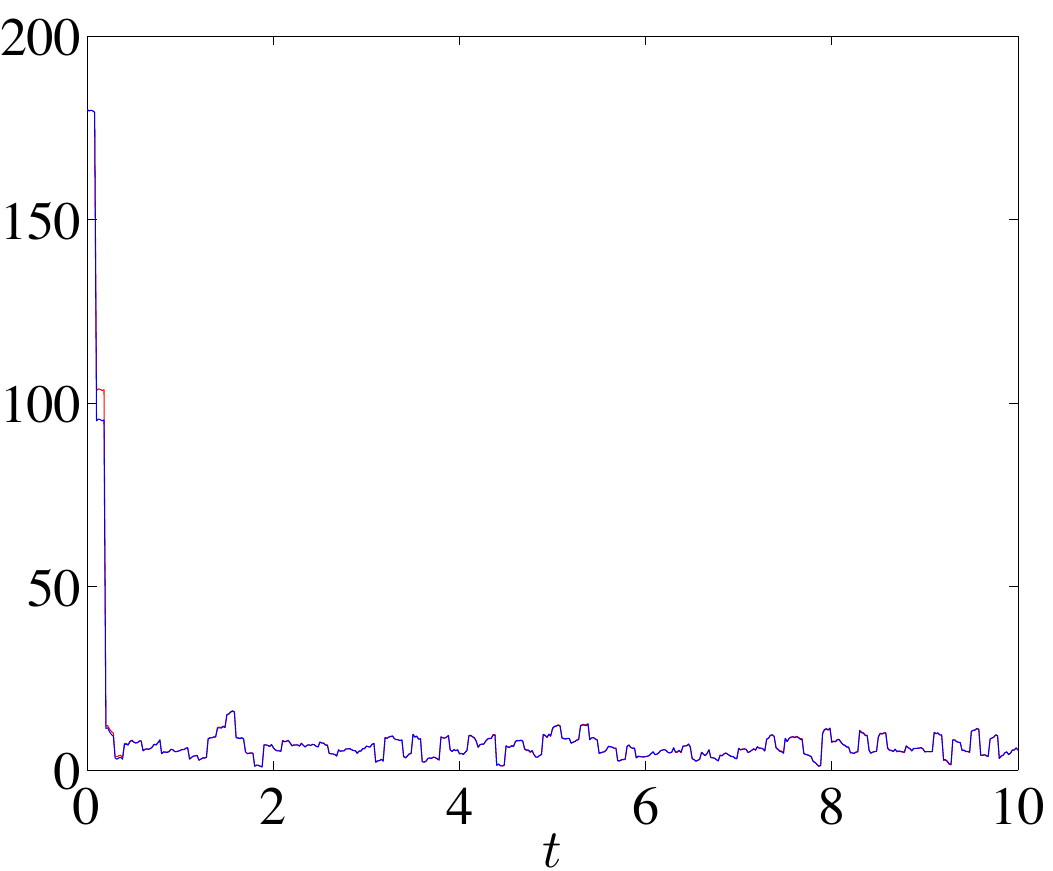}\label{fig:R_est_err_0}\hspace*{0.02\columnwidth}}
	\hfill
	\subfigure[Uncertainty measured by $\frac{1}{s_i+s_j}$]{\hspace*{0.02\columnwidth}
		\includegraphics[height=0.36\columnwidth]{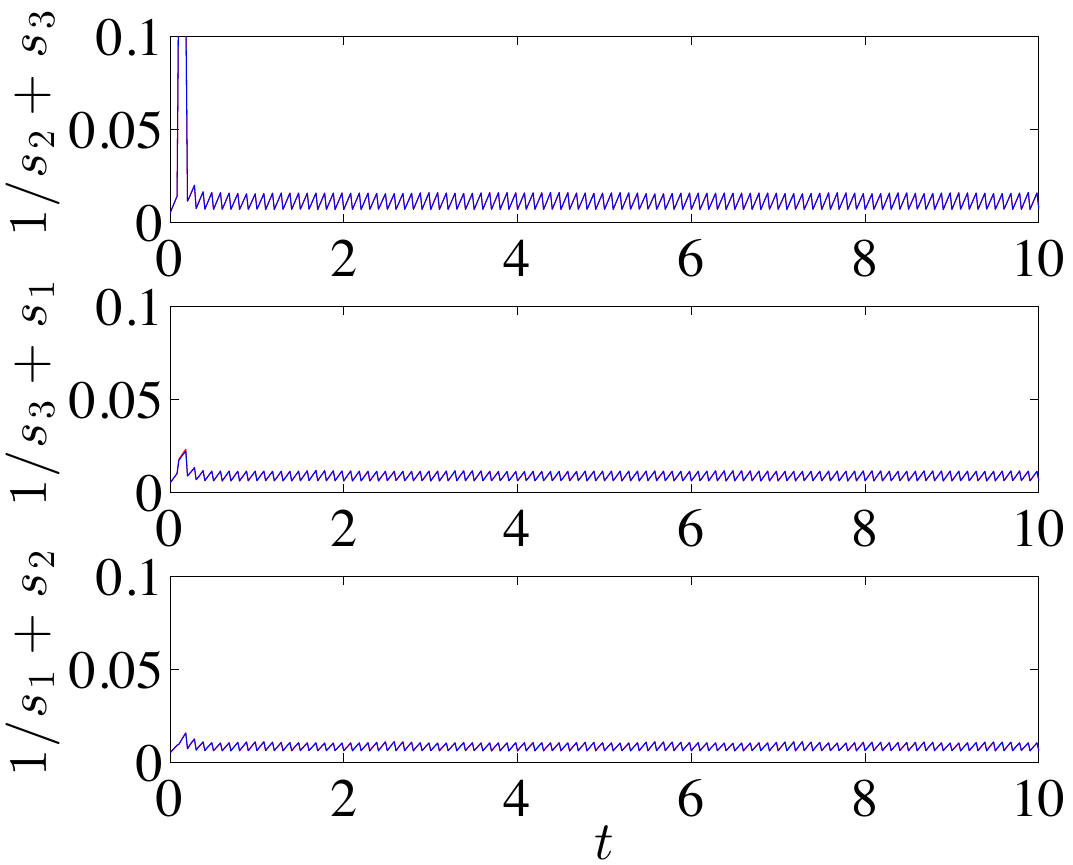}\label{fig:s_0}\hspace*{0.02\columnwidth}}
}
\caption{Case I: large initial error: first-order (red), unscented (blue)}\label{fig:ES_0}
\end{figure}
\begin{figure}
\centerline{
	\subfigure[$t=0$, $p_{\max}=1.45\times 10^4$]{\hspace*{0.065\columnwidth}
		\includegraphics[height=0.31\columnwidth]{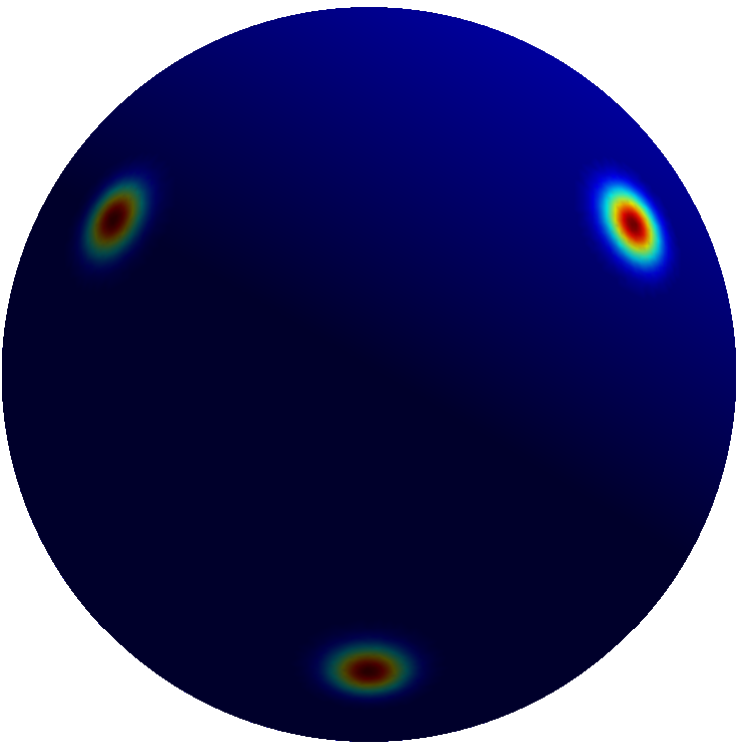}\label{fig:F0_1}\hspace*{0.065\columnwidth}}
	\hfill
	\subfigure[$t=0.08$, $p_{\max}=4.38\times 10^3$]{\hspace*{0.065\columnwidth}
		\includegraphics[height=0.31\columnwidth]{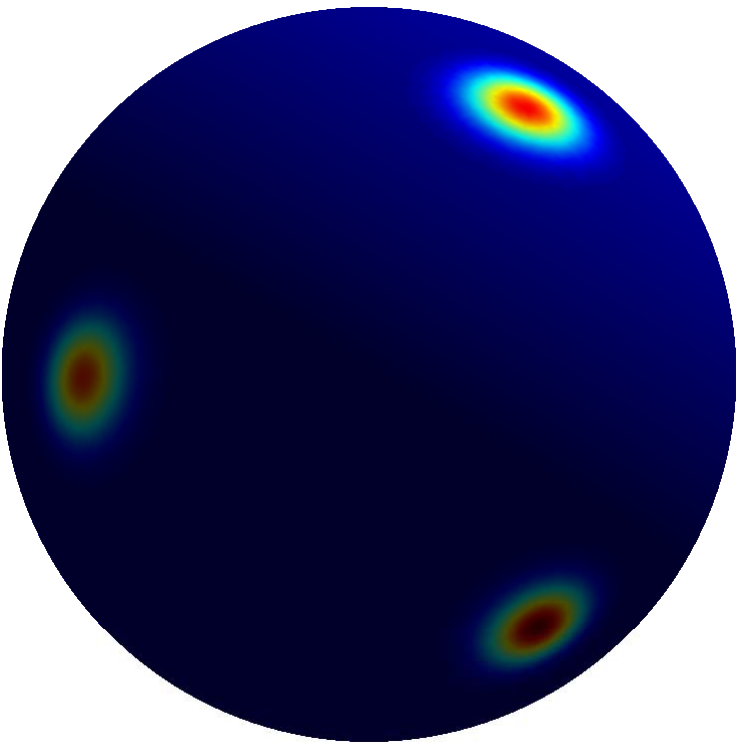}\label{fig:F0_5}\hspace*{0.065\columnwidth}}
}
\centerline{
	\subfigure[$t=0.1$, $p_{\max}=1.08\times 10^{3}$]{\hspace*{0.065\columnwidth}
		\includegraphics[height=0.31\columnwidth]{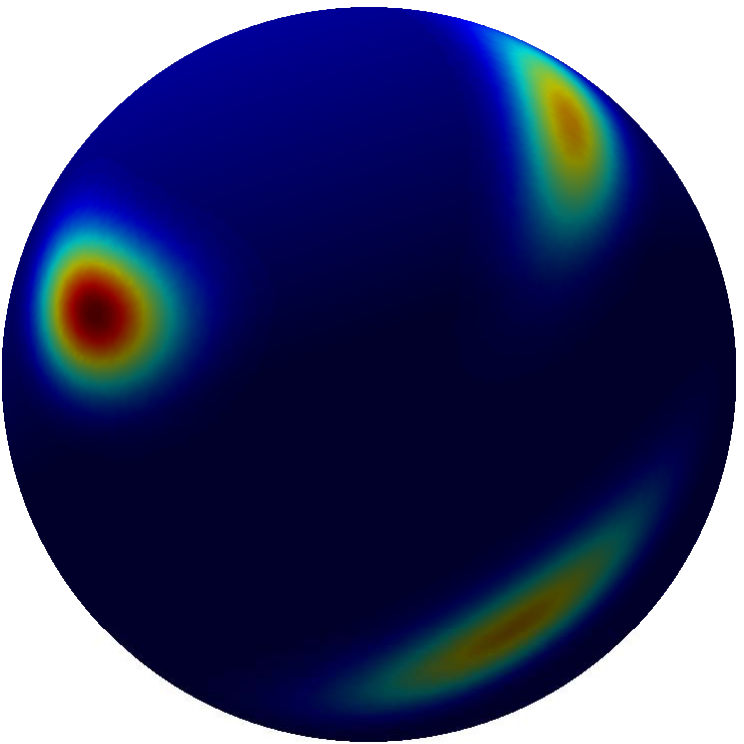}\label{fig:F0_6}\hspace*{0.065\columnwidth}}
	\hfill
	\subfigure[$t=0.3$, $p_{\max}=8.73\times 10^3$]{\hspace*{0.065\columnwidth}
		\includegraphics[height=0.31\columnwidth]{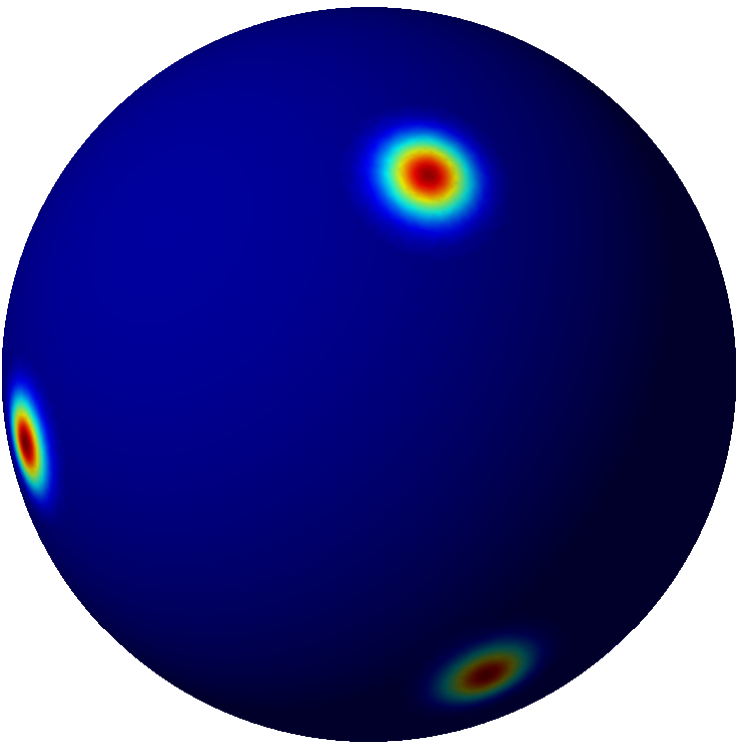}\label{fig:F0_16}\hspace*{0.065\columnwidth}}
}
\centerline{
	\subfigure[$t=1$, $p_{\max}=9.61\times 10^3$]{\hspace*{0.065\columnwidth}
		\includegraphics[height=0.31\columnwidth]{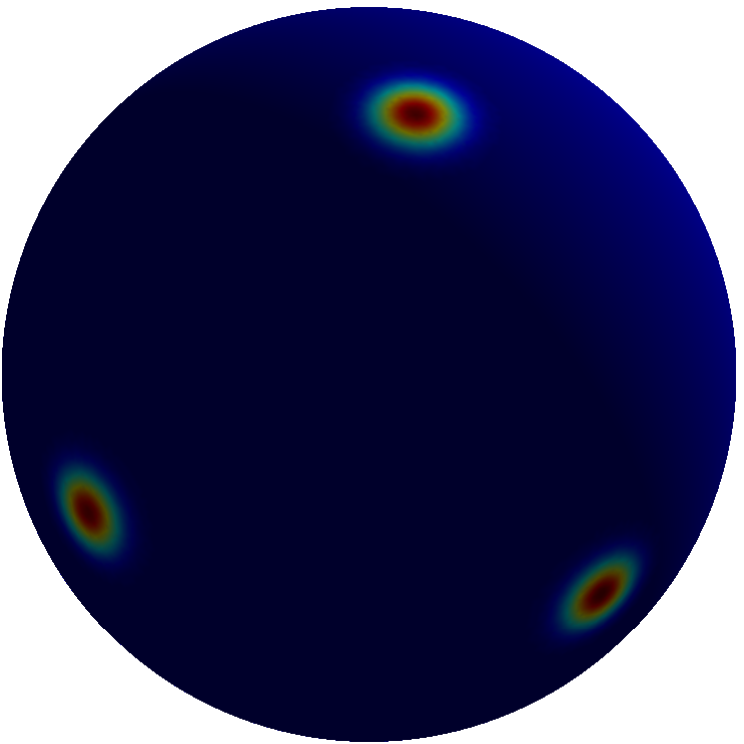}\label{fig:F0_51}\hspace*{0.065\columnwidth}}
	\hfill
	\subfigure[$t=10$, $p_{\max}=9.82\times 10^3$]{\hspace*{0.065\columnwidth}
		\includegraphics[height=0.31\columnwidth]{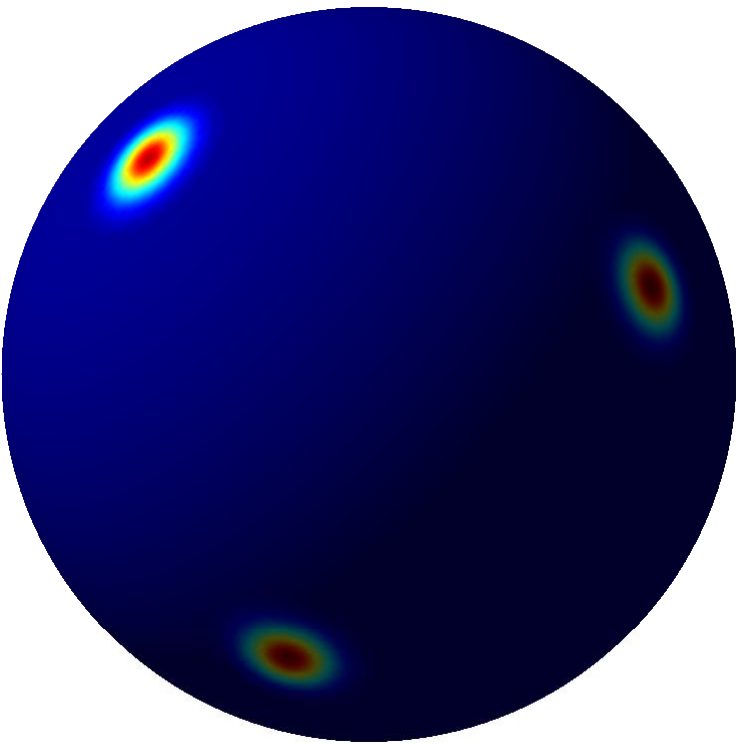}\label{fig:F0_501}\hspace*{0.065\columnwidth}}
}
\caption{Case I: visualizations of $\mathcal{M}(F_k)$ for the first-order filter}\label{fig:visES_0}
\end{figure}

These can be also observed from the visualization of $\mathcal{M}(F_k)$ of the first-order filter  in \reffig{visES_0}. Since the color shading of the figures is reinitialized in each figure, the value of the maximum probability density, corresponding to the dark red color, is specified as well. Initially, the probability distribution is highly concentrated, and it becomes dispersed a little at $t=0.08$ due to the angular velocity measurement noise. But, after the initial attitude measurement is incorporated at $t=0.1$, the probability distributions for the second axis and the third axis become dispersed noticeably due to the conflict between the belief and the measurement. This is continued until $t=0.3$. But, later at $t=1$ and $t=10$, the estimated attitude distribution becomes concentrated about the true attitude.

\begin{figure}
\centerline{
	\subfigure[Attitude estimation error ($\mathrm{deg}$)]{\hspace*{0.02\columnwidth}
		\includegraphics[height=0.36\columnwidth]{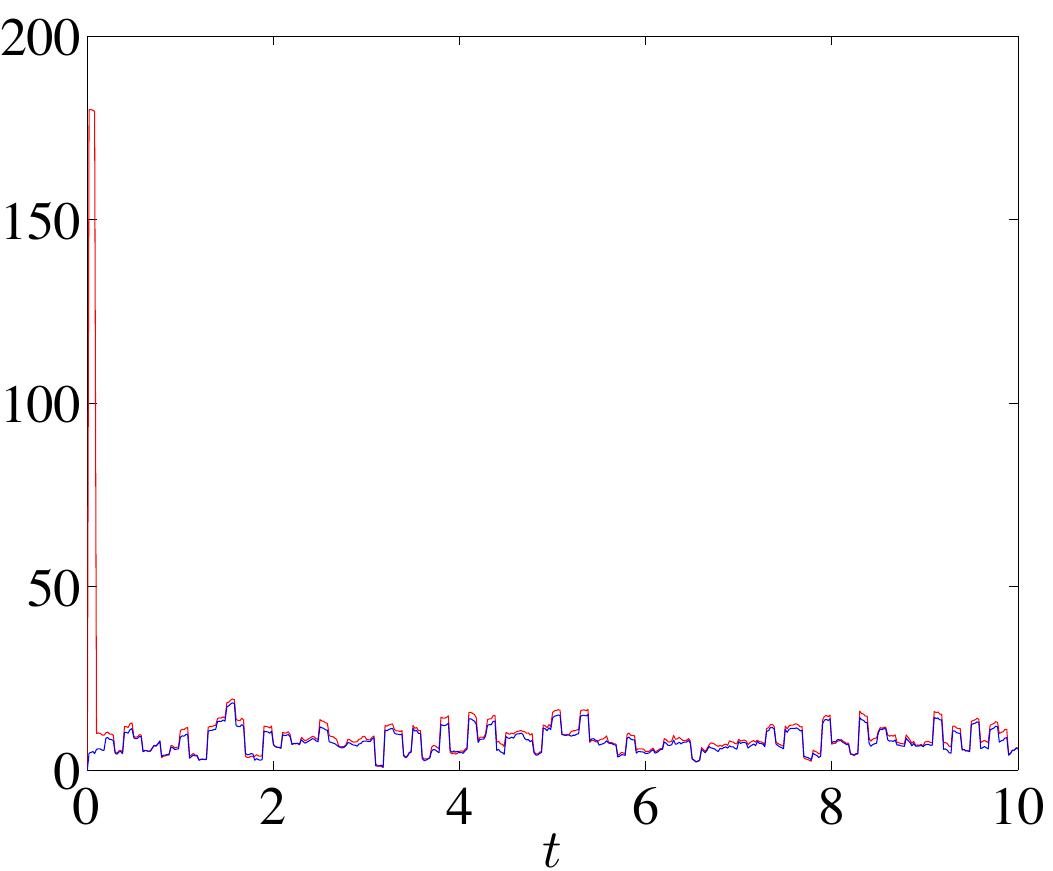}\label{fig:R_est_err}\hspace*{0.02\columnwidth}}
	\hfill
	\subfigure[Uncertainty measured by $\frac{1}{s_i+s_j}$]{\hspace*{0.02\columnwidth}
		\includegraphics[height=0.36\columnwidth]{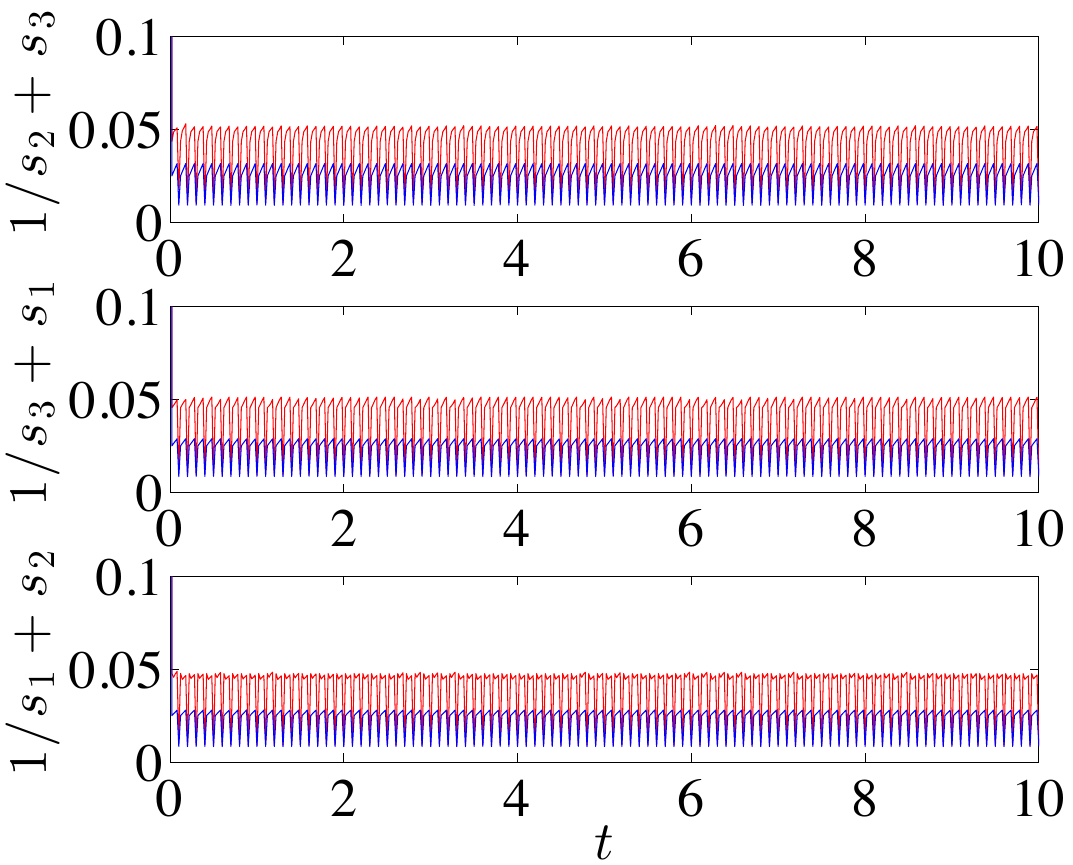}\label{fig:s}\hspace*{0.02\columnwidth}}
}
\caption{Case II: large initial uncertainty: first-order (red), unscented (blue)}\label{fig:ES_1}
\end{figure}
\begin{figure}
\centerline{
	\subfigure[$t=0$, $p_{\max}=1.00\times 10^0$]{\hspace*{0.065\columnwidth}
		\includegraphics[height=0.31\columnwidth]{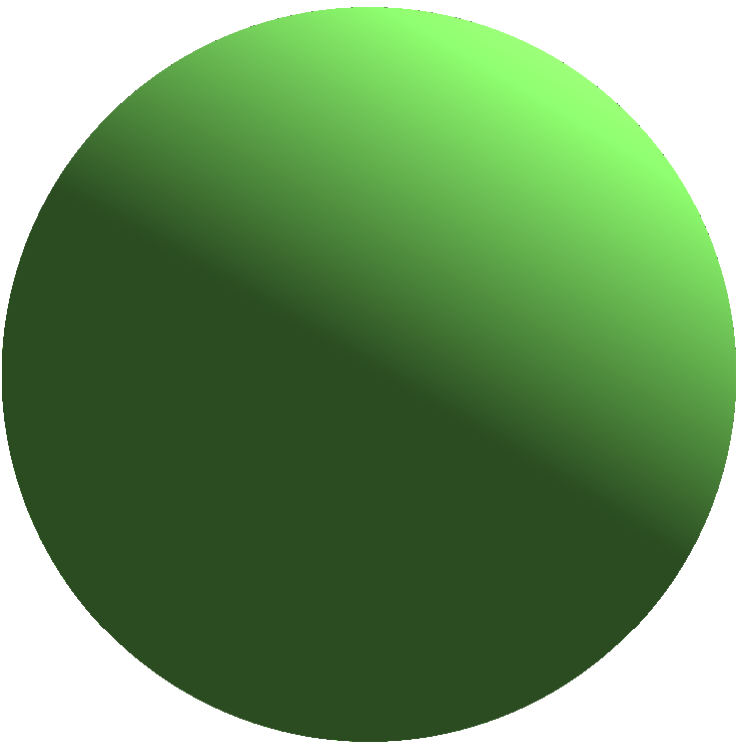}\label{fig:F_1}\hspace*{0.065\columnwidth}}
	\hfill
	\subfigure[$t=0.08$, $p_{\max}=9.83\times 10^2$]{\hspace*{0.065\columnwidth}
		\includegraphics[height=0.31\columnwidth]{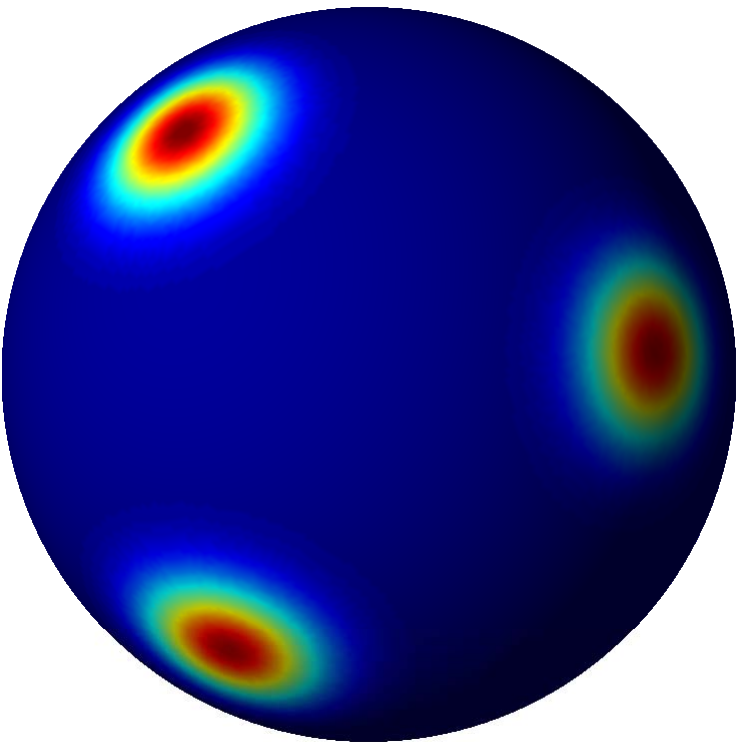}\label{fig:F_5}\hspace*{0.065\columnwidth}}
}
\centerline{
	\subfigure[$t=0.1$, $p_{\max}=6.19\times 10^{3}$]{\hspace*{0.065\columnwidth}
		\includegraphics[height=0.31\columnwidth]{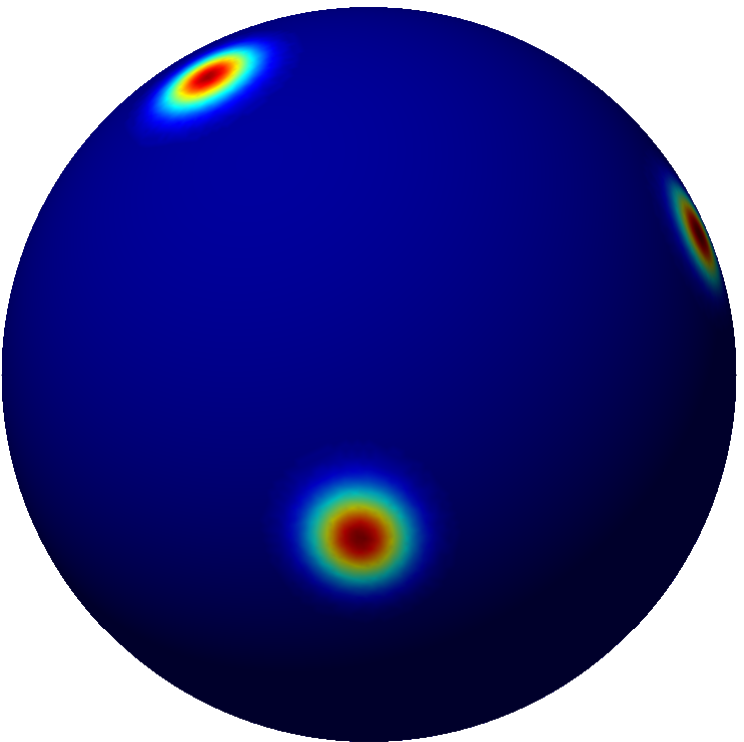}\label{fig:F_6}\hspace*{0.065\columnwidth}}
	\hfill
	\subfigure[$t=0.3$, $p_{\max}=9.83\times 10^3$]{\hspace*{0.065\columnwidth}
		\includegraphics[height=0.31\columnwidth]{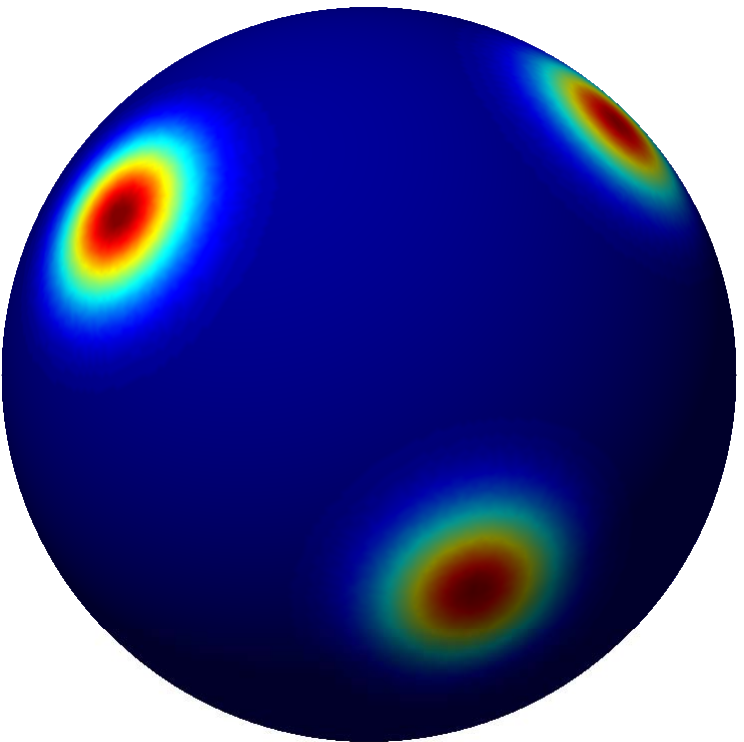}\label{fig:F_16}\hspace*{0.065\columnwidth}}
}
\centerline{
	\subfigure[$t=1$, $p_{\max}=6.21\times 10^3$]{\hspace*{0.065\columnwidth}
		\includegraphics[height=0.31\columnwidth]{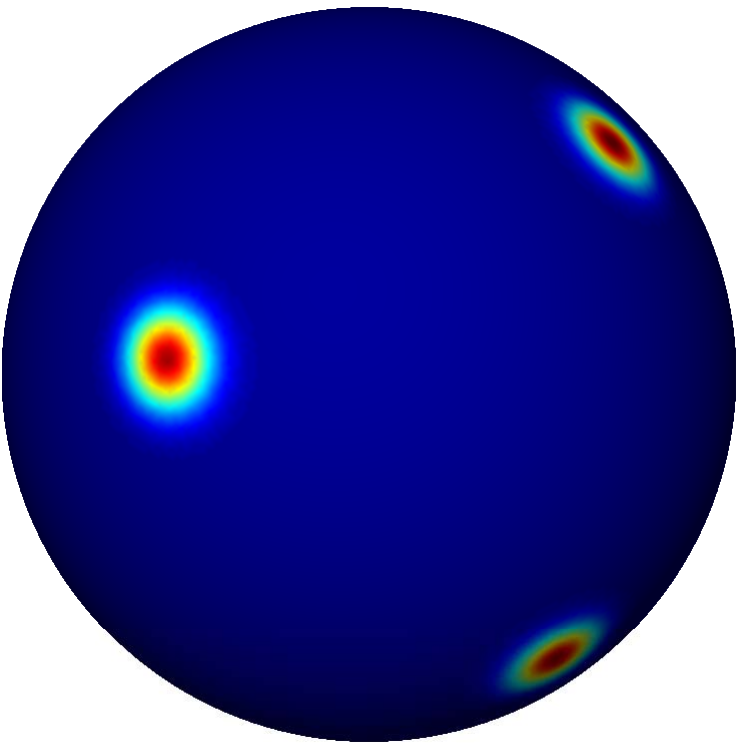}\label{fig:F_51}\hspace*{0.065\columnwidth}}
	\hfill
	\subfigure[$t=10$, $p_{\max}=6.20\times 10^3$]{\hspace*{0.065\columnwidth}
		\includegraphics[height=0.31\columnwidth]{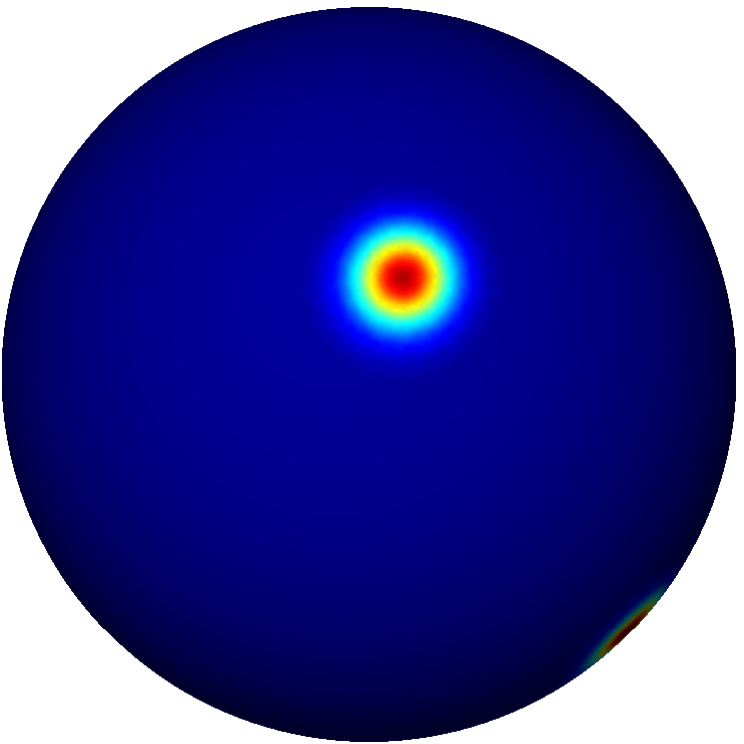}\label{fig:F_501}\hspace*{0.065\columnwidth}}
}
\caption{Case II: visualizations of $\mathcal{M}(F_k)$ for the unscented filter}\label{fig:visES_1}
\end{figure}

\subsection{Case II: Large Initial Uncertainty}\label{sec:NE2}

For the second case, the matrix parameter is chosen as
\begin{align*}
F(0)= 0_{3\times 3},
\end{align*}
which represents the uniform distribution on $\SO$, where the initial attitude is completely known. For example, this may happen when a satellite is discharged from a launch vehicle and tumbling freely in space. 

The corresponding numerical simulation results are summarized in Table \ref{tab:EST}, and Figures \ref{fig:ES_1} and \ref{fig:visES_1}. In this case, the unscented filter performs better than the first-order filter noticeably. Both the attitude estimation error and the uncertainty decrease over time, since there is no strong conflict between the measurement and the estimate as opposed to the first case. In \reffig{visES_1}, it is illustrated that the estimated distribution becomes concentrated, especially after the first attitude measurement is available at $t=0.1$. 

The presented cases for attitude estimation are particularly challenging since (i) the estimator is initially strongly confident about an incorrect attitude with the maximum error $180^\circ$, or the initial attitude is completely unknown; (ii) the considered attitude dynamics is swift and complex; (iii) both attitude and angular velocity measurement errors are relatively large; (iv) the attitude measurements are infrequent. It is shown that the proposed approach exhibits satisfactory, reasonable results even for the presented challenging cases. 


\section{Conclusions}

This paper formulates a compact form of an exponential density model, referred to the matrix Fisher distribution on $\SO$, and it presents several stochastic properties. It is shown that the shape of the distribution is specified by the nine matrix parameters, through the mean attitude, the principal axes, and the dispersion along every principal axis. This is comparable to the Gaussian distribution in $\Re^3$ that is determined by nine elements of the mean and the covariance that describe the similar attributes. The matrix Fisher distribution on $\SO$ is rich enough to characterize critical information of probability distributions on $\SO$, while being concise, especially compared with non-commutative harmonic analysis. As such, there is a great potential in the proposed matrix Fisher distribution for various stochastic analyses formulated globally on $\SO$.

In particular, two types of Bayesian attitude estimators are constructed utilizing the matrix Fisher distribution, either by the first moment matching or the unscented transform. It is illustrated that the proposed attitude estimator resolves several issues of the existing attitude estimators developed in terms of quaternions, such as singularities and ambiguities. Furthermore, they exhibit excellent convergence properties especially for challenging cases with large initial errors or uncertainties, including the uniform distributions.

\appendix

\subsection{Proof of Theorem \ref{thm:C}}\label{sec:PfC}

\subsubsection*{Properties (i-iii)}

Substituting \refeqn{USV} into \refeqn{cF},
\begin{align*}
c(F) = \int_{\SO} \exp(\trs{SU^T R V}) dR. 
\end{align*}
We perform a change of variable from $R$ to $Q=U^T RV\in\SO$. This transforms $\SO$ into $\SO$, and from the invariance of the Haar measure, we have $dQ=d(U^T RV)=dR$. Therefore,
\begin{align}
c(F) = \int_{\SO} \exp(\trs{SQ}) dQ=c(S),  \label{eqn:cS0}
\end{align}
which shows the first equality of (i). The next property (ii) follows directly from this as the singular values of $F$ remain unchanged by either taking the transpose or multiplying it with another rotation matrix. Let $C=[e_3,e_2,e_1]\in\SO$. Then, any circular shift of the diagonal elements of $S$ can be written as $(C^m)^T S C^m$ for an integer $m$. From (ii), $c(S)=c((C^m)^T S C^m)$, which shows the second equality of (i). 

From the definition of the scaled Haar measure and the property of the probability density, we have $\int_{\SO}dR=\int_{\SO}p(R)dR=1$. It follows that when the probability density is non-uniform, i.e., $S\neq 0_{3\times 3}$, 
\begin{align}
\min_{R\in\SO}\{p(R)\} < 1 <\max_{R\in\SO}\{p(R)\}.\label{eqn:pmm}
\end{align}
Since the normalizing constant is independent of $R$, maximizing $p(R)$ is equivalent to maximizing $\trs{F^T R} = \trs{VSU^T R}=\trs{S(U^T RV)}$. Let the rotation matrix  $U^TRV$ be parameterized by the exponential map as
\begin{align*}
U^T RV=\exp(\theta\hat a)=I_{3\times 3}+\sin \theta\hat a +(1-\cos\theta)\hat a^2,
\end{align*}
for $\theta\in[0,2\pi)$ and $a\in\Sph^2$. 
Substituting this,
\begin{align}
\trs{S (U^T R V)} 
& = \trs{S} - (1-\cos\theta)\sum_{(i,j,k)\in\mathcal{I}}(s_i+s_j)a_k^2.\label{eqn:trSQ}
\end{align}
Since $s_1\geq s_2\geq |s_3|\geq 0$, $s_i+s_j\geq 0$ for any $(i,j,k)\in\mathcal{I}$, $\trs{S (U^T R V)}$ is maximized when $\theta=0$, or equivalently $U^T RV=I_{3\times 3}$ and $R=UV^T$. The corresponding maximum value of the probability density is $\frac{1}{c(S)}\exp(\trs{S})$, which follows $c(S)< \exp(\trs{S})$ via \refeqn{pmm}. Similarly, \refeqn{trSQ} is minimized when $\theta=\pi$ and $a=[0,0,1]$, since $s_1+s_2\geq s_1+s_3\geq s_2+s_3$, and the minimum value of the density is $\frac{1}{c(S)}\exp(\trs{S}-2(s_1+s_2))$ which is less than 1. When $S=0_{3\times 3}$, $c(S)=1$, and the property (iii) becomes trivial. These show (iii).

\subsubsection*{Properties (iv)}

Next, consider (iv). Let $x\in\Sph^3$ be the quaternion of the rotation matrix $Q$, where the three-sphere is denoted by $\Sph^3=\{x\in\Re^4\,|\, \|x\|=1\}$. Let $q\in\Re^3$ and $q_4\in\Re$ be the vector part and the scalar part of the quaternion $x$, i.e., $x=[q^T,q_4]^T$.  The corresponding rotation matrix is written as
\begin{align}
Q(x) = (q_4^2-q^Tq)I + 2qq^T + 2q_4\hat q, \label{eqn:Qx}
\end{align}
(see, for example, \cite{ShuJAS93}). Using several properties of the trace, 
\begin{align*}
\trs{S Q(x)} & = \trs{(q_4^2 - q^T q) S + 2qq^TS}
 = x^T B x,
\end{align*}
where the diagonal matrix $B\in\Re^{4\times 4}$ is given by
\begin{align}
B= \begin{bmatrix} 2S-\trs{S}I & 0_{3\times 1} \\ 0_{1\times 3} & \trs{S} \end{bmatrix}.\label{eqn:B}
\end{align}

Substituting this into \refeqn{cS0}, and by changing the integration variable, 
\begin{align}
c(S) = \int_{\mathsf{RP}^3} \exp(x^T B x) \mathcal{J}(x) dx, \label{eqn:cS1}
\end{align}
where the real projective space is defined as $\mathsf{RP}^3=\{x\in\Sph^3\,|\, x=-x\}$. This corresponds to $\Sph^3$ where the antipodal points are identified, and it is diffeomorphic to $\SO$ via \refeqn{Qx}, i.e., $Q(\mathsf{RP}^3)=\SO$. The scalar $\mathcal{J}(x)\in\Re$ is composed of two factors. The first one is to convert the three dimensional infinitesimal volume $dx$ on $\Sph^3$ to the three dimensional volume $Q(dx)$ on $\SO$, and the second factor accounts that $dx$ and $dQ$ are normalized by the volume of $\Sph^3$ and $\SO$ respectively. 

In \cite{ShuJAS93}, the perturbation of $Q(x)$ is given by
\begin{align*}
(Q^T \delta Q)^\vee & = 2(q_4 \delta q - q \delta q_4 - q\times \delta q) = J(x) \delta x,
\end{align*}
where the matrix $J\in\Re^{3\times 4}$ is 
\begin{align*}
J(x)=2\begin{bmatrix} q_4 I-\hat q & -q \end{bmatrix}.
\end{align*}
Therefore, the scaling factor is 
\begin{align*}
\mathcal{J}(x) = \frac{2\pi^2}{8\pi^2}\sqrt{\mathrm{det}[J(x)J(x)^T]} = 
\frac{2\pi^2}{8\pi^2}\sqrt{\mathrm{det}[4I_{3\times 3}}]=2,
\end{align*}
where $2\pi^2$ and $8\pi^2$ correspond to the volume of $\Sph^3$ and $\SO$, respectively. Furthermore, the three-sphere can be considered as $\Sph^3 = \{x,-x\,|\, x\in\mathsf{RP}^3\}$, and $x^T B x$ is an even function of $B$. Applying these to \refeqn{cS1},
\begin{align}
c(S) &= \int_{\mathsf{RP}^3} 2\exp(x^T B x)\, dx 
=\int_{\Sph^3} \exp(x^T B x)\, dx.\label{eqn:cS2}
\end{align}

The above expression is equivalent to the normalizing constant of the Bingham distribution on $\Sph^3$, which is given by the hypergeometric function of matrix argument, ${}_1 F_1^{(2)} (\frac{1}{2},2;B)$~\cite{MarJup99}. In~\cite{KunSchMG04,WooAJS93}, it is shown that the hypergeometric function can be evaluated by
\begin{align}
{}_1 F_1^{(2)}& (\frac{1}{2},2;B) 
 = \int_{-1}^1 \frac{1}{2}I_0\!\bracket{\frac{1}{4}(b_2-b_1)(1-u)}\nonumber\\
&\times  I_0\!\bracket{\frac{1}{4}(b_4-b_3)(1+u)} \exp\braces{-\frac{1}{2}(b_1+b_2)u}\,du,\label{eqn:b}
\end{align}
where $b_i$ denotes the $i$-th diagonal element of $B$ for $i\in\{1,2,3,4\}$.


In short, the normalizing constant for the matrix Fisher distribution on $\SO$ corresponds to the normalizing constant for the Bingham distribution on $\Sph^3$, when the matrix $B$ is defined by \refeqn{B}. The certain equivalence between the matrix Fisher distribution and the Bingham distribution was identified in~\cite{PreJRSSS86}, and an expression for $c(S)$ is presented in~\cite{WooAJS93} based on the relation. However, that reference did not consider the scaling factor $\mathcal{J}(x)$ properly, thereby presenting an erroneous expression.

Here, combining \refeqn{b} with \refeqn{B} yield \refeqn{cS} when $(i,j,k)=(1,2,3)$. From the property (i), \refeqn{cS} is satisfied for any $(i,j,k)\in\mathcal{I}$. This shows (iv).

\subsubsection*{Properties (v-vi)} It is straightforward to obtain \refeqn{dcS} by taking the derivatives of \refeqn{cS} with each of $s_i,s_j,s_k$ using \refeqn{I1} and the chain rule. 

Next, adding \refeqn{dcSb} with \refeqn{dcSa},
\begin{align*}
\deriv{c(S)}{s_i}&+\deriv{c(S)}{s_j}=\int_{-1}^1
\frac{1}{2}(1+u)I_0\!\bracket{\frac{1}{2}(s_i-s_j)(1-u)}\nonumber\\
&\quad \times  I_1\!\bracket{\frac{1}{2}(s_i+s_j)(1+u)}\exp (s_ku)\,du.
\end{align*}
Recall $I_0(x)\geq 1$ for any $x\in\Re$ and $I_1(x)\geq 0$ when $x\geq 0$. Therefore, the integrand of the above expression is greater than or equal to zero, and this shows \refeqn{dcSij}.
Similarly, subtracting \refeqn{dcSb} from \refeqn{dcSa},
\begin{align*}
\deriv{c(S)}{s_i}&-\deriv{c(S)}{s_j}=\int_{-1}^1
\frac{1}{2}(1-u)I_1\!\bracket{\frac{1}{2}(s_i-s_j)(1-u)}\nonumber\\
&\quad \times  I_0\!\bracket{\frac{1}{2}(s_i+s_j)(1+u)}\exp (s_ku)\,du,
\end{align*}
which is greater than or equal to zero when $s_i\geq s_j$. This follows $\deriv{c(S)}{s_1}\geq\deriv{c(S)}{s_2}\geq \deriv{c(S)}{s_3}$. Find $\deriv{c(S)}{s_2}$ from \refeqn{dcSa} with $(i,j,k)=(2,3,1)$.  Since $0\leq s_2-s_3,s_2+s_3$, the integrand of \refeqn{dcSa} is greater than or equal to zero, and therefore $\deriv{c(S)}{s_2}\geq 0$. Combined with \refeqn{dcSij}, these show \refeqn{dcS12}.

The remaining property (vi) directly follows from \refeqn{dcSc}, using the chain rule and \refeqn{I1}.

\subsection{Magnus Expansion}

\begin{theorem}[Magnus Expansion~\cite{MagCPAM54}]\label{thm:Magnus}
Consider a matrix differential equation
\begin{align}
\dot Y(t) = Y(t) A(t),\label{eqn:Y_dot}
\end{align}
with $Y(0)=Y_0\in\Re^{n\times n}$. The solution can be written as
\begin{align}
Y(t) = Y_0 \exp (X(t)),\label{eqn:Y}
\end{align}
with the matrix valued function $X(t)$ defined by
\begin{align}
\dot X(t) = d\exp^{-1}_{-X(t)} ( A(t)),\quad X(0)=0,
\end{align}
where $d\exp$ denote the inverse of the derivative of the matrix exponential. Applying Picard fixed point iteration yields
\begin{align}
X(t) = \int_0^t A(\tau) d\tau + \frac{1}{2}\int_0^t \bracket{ \int_0^\tau A(\sigma)d\sigma, A(\tau)} d\tau + \mathcal{O}(t^3),\label{eqn:X}
\end{align}
where $[A,B]=AB-BA$ denotes the adjoint operator.
\end{theorem}
The Magnus expansion provides a series form of a solution \refeqn{Y} using the exponential map for a matrix differential equation given by \refeqn{Y_dot}. While the Magnus expansion is developed for a slightly different form of a matrix differential equation, namely $\dot Y(t) = A(t) Y(t)$ in~\cite{MagCPAM54,BlaCasPR09}, it is straightforward to generalize those results for \refeqn{Y_dot} to obtain \refeqn{Y}--\refeqn{X}. 

Using this, we find the solution of \refeqn{SDE} when $\Omega=0$ as follows.

\begin{prop}\label{prop:ERkp0}
Consider \refeqn{SDE}. Suppose $\Omega=0$. Then
\begin{align}
R_{k+1} = R_k \exp (\hat q_1 + \hat q_2 + \mathcal{O}(\|\Delta W_k\|^{3})),\label{eqn:Rkp}
\end{align}
where $\Delta W_k = W_{k+1}-W_k$, and 
\begin{align}
q_1 &= \int_{t_k}^{t_{k+1}} H(\tau) dW(\tau),\label{eqn:q1}\\
q_2 &= \frac{1}{2} \int_{t_k}^{t_{k+1}} \int_{t_k}^\tau H(\sigma) dW(\sigma)\times H(\tau) dW(\tau),\label{eqn:q2}
\end{align}
and the first moment of $R_{k+1}$ is given by
\begin{align}
\mathrm{E}[R_{k+1}]=\mathrm{E}[R_k]\braces{I_{3\times 3} +\frac{h}{2}(-\trs{G_k}I_{3\times 3}+G_k)+\mathcal{O}(h^{1.5})},\label{eqn:ERkpW0}
\end{align}
with $G_k=H_kH_k^T\in\Re^{3\times 3}$. 
\end{prop}
\begin{proof}
Considering \refeqn{Y_dot} as $Y^{-1} dY = A dt$, and using $[\hat x,\hat y]=(x\times y)^\vee$ for any $x,y\in\Re^3$, Theorem \ref{thm:Magnus} can be applied to $R^{-1}dR = (HdW)^\wedge$ to obtain \refeqn{Rkp}--\refeqn{q2}. Note the additional terms beyond $q_2$ contains triple (or higher) order integrations with respect to $dW$, and they have the order of $\Delta W_k$ tripled.

Next, we find the mean of $q_1$, $q_1q_1^T$ and $q_2$ as follows. Let the time interval $[t_k,t_{k+1}]$ be evenly decomposed by the sequence $\{\tau_0=t_k,\tau_1,\ldots,\tau_{N_\tau}=t_{k+1}\}$ for a positive integer $N_\tau$. From the definition of Ito's integral, 
\begin{align*}
q_1 = \lim_{N_\tau\rightarrow\infty} \sum_{l=0}^{N_\tau-1} H(\tau_l)\Delta W(\tau_l),
\end{align*}
where $\Delta W(\tau_l)=W(\tau_{l+1})-W(\tau_l)$. According to the definition of the Wiener process~\cite{KarShr91,Oks14}, the increment of the Wiener process has the zero mean. Therefore, $\mathrm{E}[q_1]=0$. Also,
\begin{align*}
q_1q_1^T = \lim_{N_\tau\rightarrow\infty} \sum_{l,m=0}^{N_\tau-1}
H(\tau_l)\Delta W(\tau_l)\Delta W(\tau_m)^T H(\tau_m)^T.
\end{align*}
From the variance of the increment of the Wiener process, $\mathrm{E}[\Delta W(\tau_l)\Delta W(\tau_m)^T]=\delta_{l,m} (\tau_{l+1}-\tau_l)I_{3\times 3}$, where $\delta_{l,m}$ denotes the Kroneker delta. Thus,
\begin{align*}
\mathrm{E}[q_1q_1^T]&=\lim_{N_\tau\rightarrow\infty} \sum_{l=0}^{N_\tau-1}
H(\tau_l)H(\tau_l)^T (\tau_{l+1}-\tau_l)\nonumber\\
&=\int_{t_k}^{t_{k+1}} H(\tau)H(\tau)^T d\tau=hG_k+\mathcal{O}(h^2).
\end{align*}

For each $l$, let the interval $[t_k,\tau_l]$ be evenly divided by the sequence $\{\sigma_0=t_k,\sigma_1,\ldots,\sigma_{N_\sigma}=\tau_l\}$ for a positive integer $N_\sigma>1$. From \refeqn{q2},
\begin{align*}
2q_2 & = \lim_{N_\tau\rightarrow\infty}\sum_{l=0}^{N_\tau-1}
\braces{\lim_{N_\sigma\rightarrow\infty}\sum_{m=0}^{N_\sigma-1}H(\sigma_m(l))\Delta W(\sigma_m(l))}\nonumber\\
&\quad\times H(\tau_l) \Delta W(\tau_l). 
\end{align*}
From the definition of the cross product and since $H$ is diagonal, the $i$-th element of $H(\sigma_m)\Delta W(\sigma_m)\times H(\tau_l) \Delta W(\tau_l)$ is given by
\begin{align*}
e_i^T& \{H(\sigma_m)\Delta W(\sigma_m)\times H(\tau_l) \Delta W(\tau_l)\}\\
&=H_j(\sigma_m)\Delta W_j(\sigma_m)
H_k(\tau_l) \Delta W_k(\tau_l)\\
&\quad -
H_k(\sigma_m)\Delta W_k(\sigma_m)
H_j(\tau_l) \Delta W_j(\tau_l),
\end{align*}
for $(i,j,k)\in\mathcal{I}$. Since $\Delta W_j$ and $\Delta W_k$ are mutually independent, and $\mathrm{E}[\Delta W_j]=\mathrm{E}[\Delta W_k]=0$, the mean of the above expression is zero, as well as $q_2$. In short, $\mathrm{E}[q_1]=\mathrm{E}[q_2]=0$, and $\mathrm{E}[q_1q_1^T]=hG_k +\mathcal{O}(h^2)$. 

Expanding the exponential map of \refeqn{Rkp},
\begin{align*}
R_{k+1}=R_k(I_{3\times 3} + \hat q_1 + \hat q_2 + \frac{1}{2}\hat q_1^2+\mathcal{O}(\|\Delta W_k\|^3)).
\end{align*}
Taking the mean, the first moment is given by
\begin{align*}
\mathrm{E}[R_{k+1}]=\mathrm{E}[R_k](I_{3\times 3} + \frac{1}{2}\mathrm{E}[\hat q_1^2]+\mathcal{O}(h^2)+\mathcal{O}(\mathrm{E}[\|\Delta W_k\|^3])).
\end{align*}
Since $\hat x^2 = xx^T -\|x\|^2 I_{3\times 3}$ for any $x\in\Re^3$,  $\mathrm{E}[\hat q_1^2]= \mathrm{E}[q_1q_1^T - \|q_1\|^2 I_{3\times 3}]= hG_k - \trs{hG_k}I_{3\times 3}$. Also as $\mathrm{E}[dW^2]=dt$ for a Wiener process $W$, $\mathcal{O}(\mathrm{E}[\|\Delta W_k\|^3]))$ is considered as  $\mathcal{O}(h^{1.5})$. These yield \refeqn{ERkpW0}.
\end{proof}

\subsection{Numerical Implementation}

When the proper singular values are large, one may encounter numerical overflow in computing the normalizing constant $c(S)$ via \refeqn{cS}. In this appendix, we introduce an  exponentially scaled normalizing constant to implement the proposed algorithms in a numerically robust fashion. A software package composed with Matlab is available in~\cite{LeeGit17}.

In  several numerical computing libraries to compute the modified Bessel function of the first kind, there is a function available to compute an exponentially scaled value~\cite{Amo}. More specifically, the exponentially scaled modified Bessel functions of the first kind are defined as
\begin{gather}
\bar I_0(x) = \exp(-|x|)  I_0(x),\label{eqn:I0_bar}\\
\bar I_1(x) = \exp(-|x|) I_1(x).\label{eqn:I1_bar}
\end{gather}
For example, in Matlab, $I_0(x)$ is computed by the command \texttt{besseli(0,x)}, and the scaled value $\bar I_0(x)$ can be obtained by \texttt{besseli(0,x,1)} with the additional option specified by the last number one. The function $\bar I_0(x)$ is differentiable when $x\neq 0$, and from \refeqn{I1}
\begin{align}
\frac{d\bar I_0(x)}{dx} = \bar I_1(x)-\mathrm{sgn}[x]\bar I_0(x).\label{eqn:dI0_bar}
\end{align}

Motivated by these, we define the \textit{exponentially scaled} normalizing constant as follows.
\begin{definition}
For a matrix Fisher distribution $\mathcal{M}(F)$, let the proper singular value decomposition is given by \refeqn{USVp}. Its exponentially scaled normalizing constant is defined as
\begin{equation}
\bar c(S) = \exp(-\trs{S}) c(S).\label{eqn:c_bar}
\end{equation}
\end{definition}

We can show that the exponentially scaled normalizing constant and its derivatives are written in terms of the scaled modified Bessel functions as summarized below.
\begin{prop}\label{thm:C_bar}
Suppose the proper singular value decomposition of $F$ is given by \refeqn{USVp}. The exponentially scaled normalizing constant for the matrix Fisher distribution \refeqn{c_bar} satisfies the following properties for any $(i,j,k)\in\mathcal{I}$.
\begin{itemize}
\item[(i)] $\bar c(S)$ is evaluated by
\begin{align}
&\bar c(S)=\nonumber\\
&\int_{-1}^1 \frac{1}{2}\bar I_0\!\bracket{\frac{1}{2}(s_i-s_j)(1-u)} 
\bar I_0\!\bracket{\frac{1}{2}(s_i+s_j)(1+u)}\nonumber\\
& \quad \times \exp ((\min\{s_i,s_j\}+s_k)(u-1))\,du.\label{eqn:c_bar_1}
\end{align}
\item[(ii)] The first order derivatives of $\bar c(S)$ are given by
\begin{align}
&\deriv{\bar c(S)}{s_k}=\nonumber\\
&\int_{-1}^1 \frac{1}{2}\bar I_0\!\bracket{\frac{1}{2}(s_i-s_j)(1-u)} 
\bar I_0\!\bracket{\frac{1}{2}(s_i+s_j)(1+u)}\nonumber\\
& \quad \times \exp ((\min\{s_i,s_j\}+s_k)(u-1))(u-1)\,du.
\label{eqn:dc_bar}
\end{align}
\item[(iii)] The second order derivatives of $\bar c(S)$ with respect to $s_k$ are given by
\begin{align}
&\frac{\partial^2\bar c(S)}{\partial s_k^2} = \nonumber\\
&\int_{-1}^1 \frac{1}{2}\bar I_0\!\bracket{\frac{1}{2}(s_i-s_j)(1-u)} 
\bar I_0\!\bracket{\frac{1}{2}(s_i+s_j)(1+u)}\nonumber\\
& \quad \times \exp ((\min\{s_i,s_j\}+s_k)(u-1))(u-1)\,du.
\label{eqn:ddc_bar_kk}
\end{align}
Also, the second order mixed derivatives are
\begin{align}
&\frac{\partial^2 \bar c(S)}{\partial s_i\partial s_j}  = \nonumber\\
&\int_{-1}^1 \frac{1}{4}\bar I_1\bracket{\frac{1}{2}(s_j-s_k)(1-u)}
 \bar I_0\bracket{\frac{1}{2}(s_j+s_k)(1+u)} \nonumber\\
&\times u(1-u)\exp ((s_i+\min\{s_j,s_k\})(u-1))\nonumber\\
& + \frac{1}{4}\bar I_0\bracket{\frac{1}{2}(s_j-s_k)(1-u)}  
\bar I_1\bracket{\frac{1}{2}(s_j+s_k)(1+u)}\nonumber\\
&\times u(1+u)\exp ((s_i+\min\{s_j,s_k\})(u-1))\,du\nonumber\\
& - \deriv{\bar c(S)}{s_i}- \deriv{\bar c(S)}{s_j}-\bar c(S).\label{eqn:ddc_bar_ij}
\end{align}
\item[(iv)] The derivatives of $c(S)$ can be rediscovered by 
\begin{align}
\deriv{c(S)}{s_i} & = e^{\trs{S}}\parenth{\bar c(S)+\deriv{\bar c(S)}{s_i}},\label{eqn:dc_dc_bar}\\
\frac{\partial^2 c(S)}{\partial s_i\partial s_j} 
& = e^{\trs{S}}\parenth{\bar c(S)+\deriv{\bar c(S)}{s_i}+\deriv{\bar c(S)}{s_j}+\frac{\partial^2 \bar c(S)}{\partial s_i\partial s_j}}.\label{eqn:ddc_ddc_bar}
\end{align}
\end{itemize}
\end{prop}
\begin{proof}
Substitute \refeqn{I0_bar} into \refeqn{cS}, and rearrange. When $s_i\geq s_j$, it reduces to
\begin{align*}
&c(S)  = \\
&e^{s_i} \int_{-1}^1 \frac{1}{2}\bar I_0\!\bracket{\frac{1}{2}(s_i-s_j)(1-u)} 
\bar I_0\!\bracket{\frac{1}{2}(s_i+s_j)(1+u)}\nonumber\\
&\quad \times \exp ((s_j+s_k)u)\,du,
\end{align*}
or when $s_j\geq s_i$,
\begin{align*}
& c(S)  = \\
& e^{s_j} \int_{-1}^1 \frac{1}{2}\bar I_0\!\bracket{\frac{1}{2}(s_i-s_j)(1-u)} 
\bar I_0\!\bracket{\frac{1}{2}(s_i+s_j)(1+u)}\nonumber\\
&\quad \times \exp ((s_i+s_k)u)\,du,
\end{align*}
for any $(i,j,k)\in\mathcal{I}$. From \refeqn{c_bar}, these show \refeqn{c_bar_1}. Taking the derivatives of \refeqn{c_bar_1} with respect to $s_k$, it is straightforward to show \refeqn{dc_bar}, \refeqn{ddc_bar_kk}.

Next, taking the derivatives of \refeqn{dc_bar} with respect to $s_i$ is cumbersome as $\bar I_0(x)$ is not differentiable at $x=0$. Instead, substitute \refeqn{I0_bar} and \refeqn{I1_bar} to \refeqn{ddcS} to obtain
\begin{align*}
&\frac{\partial^2 c(S)}{\partial s_i\partial s_j}  = e^{\trs{S}}\nonumber\\
&\int_{-1}^1 \frac{1}{4}\bar I_1\bracket{\frac{1}{2}(s_j-s_k)(1-u)}
 \bar I_0\bracket{\frac{1}{2}(s_j+s_k)(1+u)} \nonumber\\
&\times u(1-u)\exp ((s_i+\min\{s_j,s_k\})(u-1))\nonumber\\
& + \frac{1}{4}\bar I_0\bracket{\frac{1}{2}(s_j-s_k)(1-u)}  
\bar I_1\bracket{\frac{1}{2}(s_j+s_k)(1+u)}\nonumber\\
&\times u(1+u)\exp ((s_i+\min\{s_j,s_k\})(u-1))\,du.
\end{align*}
Equations \refeqn{dc_dc_bar} and \refeqn{ddc_ddc_bar} also follow directly from \refeqn{c_bar}. Substituting the above equation to \refeqn{ddc_ddc_bar} yield \refeqn{ddc_bar_ij}.
\end{proof}

We can rewrite the stochastic properties of the matrix Fisher distribution and the proposed attitude estimation schemes in terms of the exponentially scaled normalizing constant and its derivatives. These yield more numerically robust implementation that avoid numerical overflows that possibly appear particularly when the proper singular values are large. 

First, the probability density value given by \refeqn{MF} can be rewritten in terms of the scaled normalizing constant as
\begin{equation}
p(R) = \frac{1}{\bar c(S)} \exp(\trs{F^T R}-\trs{S}).\label{eqn:MF_bar}
\end{equation}
The first and the second moments presented in Theorem \ref{thm:Q} are given by
\begin{align}
\mathrm{E}[Q_{ij}]
=\begin{cases}1+\dfrac{1}{\bar c(S)}\dderiv{\bar c(S)}{s_i}=1+\dderiv{\log \bar c(S)}{s_i} & \mbox{if $i= j$},\\
0 & \mbox{otherwise},
\end{cases}\label{eqn:M1_bar}
\end{align}
for $i,j\in\{1,2,3\}$.
\begin{align}
\mathrm{E}[Q_{ij}Q_{kl}]
=\begin{cases}1+\dfrac{1}{\bar c(S)}\parenth{\displaystyle \deriv{\bar c(S)}{s_i}+\deriv{\bar c(S)}{s_j}+\dfrac{\partial^2 \bar c(S)}{\partial s_i\partial s_j}}\\
\hfill \mbox{if $i=j$ and $k=l$},\\
0 \hfill \mbox{otherwise},
\end{cases}\label{eqn:M2_bar}
\end{align}
for $i,j,k,l\in\{1,2,3\}$. 

For the unscented transform defined at Definition \ref{def:UT},  the term $\log c(S)$ at \refeqn{costhetai} can be replaced with $\log c(S)=\trs{S}+ \log \bar c(S)$, and \refeqn{M1_bar} can be used to compute the weighting parameters \refeqn{Wi}.

Next, the implicit equation \refeqn{MLEs} can be equivalently reformulated in terms of the scaled normalizing constant as
\begin{equation}
\frac{1}{\bar c(S)}\deriv{\bar c(S)}{ s_i}=d_i-1,
\label{eqn:MLEs_bar}
\end{equation}
for all $i\in\{1,2,3\}$. This is arranged into a vector form as
\begin{equation}
f(s)=\frac{1}{\bar c(S)}\deriv{\bar c(S)}{ s}-\begin{bmatrix} d_1-1\\d_2-1\\d_3-1\end{bmatrix}=0,\label{eqn:f_cbar}
\end{equation}
which can be solved via the following Newton's iteration
\begin{equation}
s^{(q+1)}= s^{(q)}- \parenth{\deriv{f(s)}{s}\bigg|_{s=s^{(q)}}}^{-1} f(s^{(q)}),
\end{equation}
where the superscript $(q)$ denotes the number of iterations, and the gradient of \refeqn{f_cbar} is given by
\begin{equation}
\deriv{f(s)}{s} = \frac{1}{\bar c(S)}\frac{\partial^2 \bar c(S)}{\partial s^2}-\frac{1}{\bar c(S)^2}\deriv{\bar c(S)}{s}\parenth{\deriv{\bar c(S)}{s}}^T.
\end{equation}


%

\bibliography{/Users/tylee.fdcl/Documents/BibMaster17,/Users/tylee.fdcl/Documents/tylee}
\bibliographystyle{IEEEtran}

\begin{IEEEbiography}{Taeyoung Lee}
 is an associate professor of the Department of Mechanical and Aerospace Engineering at the George Washington University. He received his doctoral degree in Aerospace Engineering and his master's degree in Mathematics at the University of Michigan in 2008. His research interests include geometric mechanics and control with applications to complex aerospace systems. 
\end{IEEEbiography}
\vfill

\end{document}